\documentclass[leqno,12pt]{article}
\def\today{${\scriptscriptstyle\number\day-\number\month-\number\year}$}
\usepackage{graphicx}
\usepackage{fancyhdr}
\usepackage{amsmath}
\usepackage{amsthm}
\usepackage{amsfonts}
\newtheorem{theorem}{Theorem}[section]
\newtheorem{lemma}[theorem]{Lemma}
\theoremstyle{definition}
\newtheorem{notation}[theorem]{Notation}
\newtheorem{remark}[theorem]{Remark}

\def\address#1{{\center{#1}}}
\date{}
\def\m@th{\mathsurround=0pt}
\def\eqal#1{\null\,\vcenter{\openup\jot\m@th
 \ialign{\strut\hfil$\displaystyle{##}$&&$\displaystyle{{}##}$\hfil
 \crcr#1\crcr}}\,}
\def\matrix#1{\null\,\vcenter{\normalbaselines\m@th
 \ialign{\hfil$##$\hfil&&\quad\hfil$##$\hfil\crcr
 \mathstrut\crcr\noalign{\kern-\baselineskip}
 #1\crcr\mathstrut\crcr\noalign{\kern-\baselineskip}}}\,}
\def\N{{\mathbb N}}
\def\R{{\mathbb R}}
\def\divv{{\rm div}\,}
\def\rot{{\rm rot}\,}

\def\diagin{\hbox{---}\hskip-13.5pt\intop}
\def\diagint{\hbox{--}\hskip-8.2pt\intop}
\def\diagintop{\mathop{\mathchoice
{{\diagin}}%
{{\diagint}}%
{{\diagint}}%
{{\diagint}}%
}\limits}

\numberwithin{equation}{section}

\title{On Ladyzhenskaya-Serrin condition sufficient for regular solutions to the Navier-Stokes equations. Periodic case}
\author{Wojciech M. Zaj\c{a}czkowski}

\begin{document}
\input amssym.def
\input amssym.tex
\maketitle
%%%%%%%%%%%%%%%%%%%%%%%%%%%%%%%%%%%%%%%%%%%%%%%%%%%%%%%%%%%%%%%%%%%%%%%%%
\thispagestyle{fancy}
%%%%%%%%%%%%%%%%%%%%%%%%%%%%%%%%%%%%%%%%%%%%%%%%%%%%%%%%%%%%%%%%%%%%%%%%%

\address{Institute of Mathematics, Polish Academy of Sciences,\\
\'Sniadeckich 8, 00-656 Warsaw, Poland\\
e-mail:wz@impan.pl\\
Institute of Mathematics and Cryptology, %\\
Cybernetics Faculty, \\
Military University of Technology,\\% Warsaw\\
S. Kaliskiego 2, 00-908 Warsaw, Poland\\}

\begin{abstract}
We consider the Navier-Stokes equations in a bounded domain with periodic boundary conditions. Let $V=V(x,t)$ be the velocity of the fluid. The aim of this paper is to prove the bound $\|V(t)\|_{H^1}\le c$ for any $t\in\R_+$, where $c$ depends on data. The proof is divided into two steps. In the first step the Lam\'e system with a special version of the convective term is considered. The system has two viscosities. Assuming that the second viscosity (the bulk one) is sufficiently large we are able to prove the existence of global regular solutions to this system. The proof is divided into two steps. First the long time existence in interval $(0,T)$ is proved, where $T$ is proportional to the bulk viscosity. Having the bulk viscosity large we are able to show that data at time $T$ are sufficiently small. Then by the small data arguments a global existence follows. In this paper we are restricted to derive appropriate estimates only. To prove the existence we should use the method of successive approximations and the continuation argument. Let $v$ be a solution to it. In the second step we consider a problem for $u=v-V$. Assuming that $\|u\|_{H^1}$ at $t=0$ is sufficiently small we show that $\|u(t)\|_{H^1}$ is also sufficiently small for any $t\in\R_+$.\\
Estimates for $v$ and $u$ in $H^1$ imply estimate for $\|V(t)\|_{H^1}$ for any $t\in\R_+$.
\end{abstract}



\section{Introduction}\label{s1}

We consider the Navier-Stokes equations with the periodic boundary conditions
\begin{equation}\eqal{
&V_t+V\cdot\nabla V-\mu\Delta V+\nabla P=0,\cr
&\divv V=0,\cr
&V|_{t=0}=V_0,\cr}
\label{1.1}
\end{equation}
where $V=(V_1(x,t),V_2(x,t),V_3(x,t))\in\R^3$ is the velocity of the fluid, $x=(x_1,x_2,x_3)$ are the Cartesian coordinates, $P=P(x,t)\in\R$ is the pressure, $\mu>0$ is the viscosity coefficient. We denote the considered domain by $\Omega$.

Our aim is to prove such a priori estimate for solutions to (\ref{1.1}) which by the Ladyzhenskaya-Serrin condition the regularity of the weak solutions to (\ref{1.1}) follows.

Since we are not able to find the estimate directly for solutions to system (\ref{1.1}) we proceed as follows. Let $V_0\in H^1$. Then we take a function $v_0$ from a sufficiently small neighborhood of $V_0$ in $H^1$. Next, we assume that $v_0$ are initial data for solutions to the following problem with the periodic boundary conditions
\begin{equation}\eqal{
&v_t+v_H\cdot\nabla v-\mu\Delta v-\nu\nabla\divv v=0,\cr
&v|_{t=0}=v_0,\cr}
\label{1.2}
\end{equation}
where $v_H$ is the divergence free part of $v$ derived by the Helmholtz decomposition.

\noindent
Moreover, $\mu$, $\nu$ are positive constant coefficients. We have to emphasize that $\mu$ in (\ref{1.1}) and in (\ref{1.2}) are the same.

However, $v_0$ belongs to some neighborhood of $V_0$ in $H^1$ we shall need that it is more regular. The regularity will be derived later.

We have to emphasize that system (\ref{1.2}) appeared from discussion between the author and Piotr Mucha which we had during our stay in Prague Nov. 19--23, 2019. We have to mention that system (\ref{1.2}) replaces a system of weakly compressible Navier-Stokes equations (density is close to a constant, the second viscosity is sufficiently large so the gradient part of velocity is sufficienty small) named by WCNSE and egzamined in \cite{Z1, Z3, Z5}. The WCNSE system was used in \cite{Z2, Z4, Z6} to prove regularity of the weak solutions to the Navier-Stokes equations. WCNSE is a physical system.

\noindent
The proofs of existence of global regular solutions to WCNSE in \cite{Z1, Z3, Z5} are very complicated. However, system (\ref{1.2}) does not have a physical meaning it strongly simplifies considerations in \cite{Z1}--\cite{Z6}.

The crucial point of estimates in \cite{Z1, Z3, Z5} is the estimate for velocity $v$ in $L_\infty(0,T;L_6(\Omega))$. Since we have three different proofs the three different papers are written.

To simplify considerations we assume that
\begin{equation}
\intop_\Omega v_0dx=0.
\label{1.3}
\end{equation}
Periodic boundary conditions in (\ref{1.2}) and (\ref{1.3}) imply that
\begin{equation}
\intop_\Omega vdx=0.
\label{1.4}
\end{equation}
Moreover, (\ref{1.3}) and definition of $v_0$ give the restriction
\begin{equation}
\intop_\Omega V_0dx=0.
\label{1.5}
\end{equation}
In view of (\ref{1.4}) and the structure of system (\ref{1.2}) we are looking for solutions to (\ref{1.2}) in the form
\begin{equation}
v=\nabla\varphi+\rot\psi,
\label{1.6}
\end{equation}
where $\varphi$, $\psi$ are periodic functions. In this case $v_H=\rot\psi$ and problem (\ref{1.2}) takes the form
\begin{equation}\eqal{
&v_t+\rot\psi\cdot\nabla v-\mu\Delta v-\nu\nabla\divv v=0\cr
&v|_{t=0}=v_0.\cr}
\label{1.7}
\end{equation}
Hence, problem (\ref{1.7}) is treated as an auxiliary problem which helps us to prove regularity of solutions to the Navier-Stokes equations.

To derive a global a priori estimate guaranteeing the existence of regular solutions to problem (\ref{1.1}) we have to linearize the Navier-Stokes equations. For this purpose we use solutions to problem (\ref{1.7}) and derive a problem for the difference
\begin{equation}
u=v-V
\label{1.8}
\end{equation}
which has the following form
\begin{equation}\eqal{
&u_t+\rot\psi\cdot\nabla u+(u-\nabla\varphi)\cdot\nabla(v-u)-\mu\Delta u\cr
&\quad-\nu\nabla\divv u-\nabla P=0,\cr
&u|_{t=0}=v_0-V_0\equiv u_0,\cr}
\label{1.9}
\end{equation}
where we used (\ref{1.6}).
\goodbreak

\noindent
However, problem (\ref{1.9}) is not linear it implies that sufficient smallness of $\|u_0\|_{H^1(\Omega)}$ gives that $\|u(t)\|_{H^1(\Omega)}$ is controlled for all time. This is possible thanks to sufficient regularity of $v$ and sufficiently large $\nu$.

\begin{remark}\label{r1.1}
To justify (\ref{1.6}) we have to find elliptic problems implying existence of potentials $\varphi$ and $\psi$. Let $v$ be given. Let $\varphi$, $\psi$ satisfy the periodic boundary conditions. Then $\varphi$ and $\psi$ are solutions to the following elliptic problems
\begin{equation}\eqal{
&\Delta\varphi=\divv v,\cr
&\varphi\ {\rm satisfies\ the\ periodic\ boundary\ conditions},\cr
&\intop_\Omega\varphi dx=0\cr}
\label{1.10}
\end{equation}
and
\begin{equation}\eqal{
&\rot^2\psi=\rot v,\cr
&\divv\psi=0,\cr
&\psi\ {\rm satisfies\ the\ periodic\ boundary\ conditions},\cr
&\intop_\Omega\psi dx=0.\cr}
\label{1.11}
\end{equation}
\end{remark}
Now we collect results of this paper. The used notation is described in Notations \ref{n1.6} and \ref{n1.7}.

\noindent
From Theorem \ref{t3.6} and Lemma \ref{l4.1} we have

\begin{theorem}\label{t1.2}
Consider problem (\ref{1.7}). Assume that $v$ is expressed in the form (\ref{1.6}), where potentials $\varphi$, $\psi$ are solutions to problems (\ref{1.10}), (\ref{1.11}), respectively. Assume that $\nabla\varphi(0)\in\Gamma_1^2(\Omega)$, $\rot\psi(0)\in\Gamma_1^2(\Omega)$ and the quantity $\sqrt{\nu}\|\nabla\varphi(0)\|_{\Gamma_1^2(\Omega)}+\|\rot\psi(0)\|_{\Gamma_1^2(\Omega)}$ is bounded. Assume that there exist constants $c_3$, $c_4$ such that $c_3/\nu^\varkappa\le\|\varphi(0)\|_{L_p(\Omega)}\le c_4/\nu^\varkappa$, $p\in(3,6)$, $\varkappa=3/2-3/p\in(1/2,1)$. Moreover, $|\varphi(0)|_2\le c_4/\nu^\varkappa$.\\
Then for $\nu$ sufficiently large there exists a constant $A$ satisfying (\ref{3.54}) such that solutions to problem (\ref{1.7}) satisfy the bound
\begin{equation}\eqal{
&\sqrt{\nu}\|\nabla\varphi(t)\|_{\Gamma_1^2(\Omega)}+\|\rot\psi(t)\|_{\Gamma_1^2(\Omega)}+ \nu(\sqrt{\nu}\|\nabla\varphi\|_{L_2(0,t;\Gamma_1^3(\Omega))}\cr
&\quad+\|\rot\psi\|_{L_2(0,t;\Gamma_1^3(\Omega))})+\nu \|\nabla\varphi\|_{L_2(0,t;\Gamma_1^3(\Omega))}\le A,\cr}
\label{1.12}
\end{equation}
where $t\le T\le\nu^\beta$, $\beta<2(1-\varkappa)$.
If $cA^4\le\mu T$, where $c$ is the constant from (\ref{4.13}), then
\begin{equation}\eqal{
&\|\nabla\varphi(T)\|_{\Gamma_1^2(\Omega)}+\|\rot\psi(T)\|_{\Gamma_1^2(\Omega)}\le \|\nabla\varphi(0)\|_{\Gamma_1^2(\Omega)}\cr
&\quad+\|\rot\psi(0)\|_{\Gamma_1^2(\Omega)}.\cr}
\label{1.13}
\end{equation}
Moreover, it is shown that
\begin{equation}
\|\varphi(T)\|_3+\|v(T)\|_2\le c(e^{-\mu T}|v(0)|_2+A/\sqrt{\nu})\equiv cB.
\label{1.14}
\end{equation}
Then for $T$, $\nu$ sufficiently large the following inequality holds
\begin{equation}
\|v\|_{2,\infty,\Omega_T^t}+\|v\|_{3,2,\Omega_T^t}+ \|\varphi\|_{3,\infty,\Omega_T^t}|+\|\nabla\varphi\|_{3,2,\Omega_T^t}\le cB,
\label{1.15}
\end{equation}
for any $t\in(T,\infty)$.
\end{theorem}

\begin{theorem}\label{t1.3}
Consider problem (\ref{1.9}) with coefficients dependent on solutions to problem (\ref{1.7}). Let the assumptions of Theorem \ref{t1.2} hold. Let $u(0)\in H^1(\Omega)$ and $\|u(0)\|_{H^1(\Omega)}\le\gamma$, where $\gamma$ is sufficiently small. Then, for $\nu$ sufficiently large, we have
\begin{equation}
\|u(t)\|_{H^1(\Omega)}\le\gamma
\label{1.16}
\end{equation}
for any $t\in\R_+$.
\end{theorem}

\begin{theorem}\label{t1.4}
Let the assumptions of Theorems \ref{t1.2} and \ref{t1.3} hold. Then, by (\ref{1.12}), (\ref{1.15}) and (\ref{1.16}), we have
\begin{equation}
\|V\|_{L_\infty(\R_+;H^1(\Omega))}\le c(A+\gamma).
\label{1.17}
\end{equation}
\end{theorem}

\begin{remark}\label{r1.5}
The aim of this paper is to prove estimate (\ref{1.17}) for solutions to the Navier-Stokes equations (\ref{1.1}). We do not know how (\ref{1.17}) can be proved directly for solutions to (\ref{1.1}). Therefore, we consider first problem (\ref{1.7}), where $v$ is used in form (\ref{1.6}). Since (\ref{1.7}) is considered with large parameter $\nu$ decomposition (\ref{1.6}) implies smallness of $\varphi$ in $\Gamma_1^2(\Omega)$ if $\varphi(0)$ in $\Gamma_1^2(\Omega)$ is small too. Next, for $\nu$ sufficiently large, we are able to prove the global existence of solutions to (\ref{1.7}) without any restrictions on the magnitude of $\psi$ in suitable norms. It is important for comparison $V$ with $v$ because the magnitude of divergence free $V$ can be restricted by data only. This implies that function $u$ defined by (\ref{1.8}) and satifying (\ref{1.9}) can be show small in $H^1(\Omega)$ if the difference $\rot\psi|_{t=0}-V|_{t=0}$ in $H^1(\Omega)$ and $\|\nabla\varphi(0)\|_1\le c/\nu$ are small too. We have to mention that $\nu$ is large but always finite. Any passage with $\nu$ to infinity makes the considerations in this paper wrong.

The first step in the proof of global existence of solutions to problem (\ref{1.7}) is derivation of estimate (\ref{1.12}). The proof of (\ref{1.12}) is divided into three main steps. First we derive inequality (\ref{2.27}) of the form
\begin{equation}
|v|_{6,\infty,\Omega^t}\le D_1,
\label{1.18}
\end{equation}
where $D_1$ is an increasing positive function of $\Psi/\nu^\alpha$, $\alpha>0$.

\noindent
Inequality (\ref{1.18}) follows from (\ref{2.2}) and (\ref{1.19}), where we needed that \break $c_1/\nu^\varkappa\le|\varphi(0)|_p\le c_2/\nu^\varkappa$, $c_1<c_2$, $\varkappa\in(1/2,1)$, $p\in(3,6)$. We have to emphasize that (\ref{1.18}) is not any estimate because $D_1$ depends on some norms of $v$ multiplied by $\Psi/\nu^\alpha$. Moreover, $t\le\nu^\beta$, $\beta<2(1-\varkappa)$ yields that time $t$ does not appear in $D_1$.

\noindent
In the next step we derive inequality (\ref{2.28}) of the form
\begin{equation}
|v_t(t)|_2+\|v_t\|_{1,2,\Omega^t}\le D_2,
\label{1.19}
\end{equation}
where $D_2$ is an increasing function of $D_1$.

Denote the l.h.s. of (\ref{1.12}) by $X(t)$.

\noindent
In the third step by applying the energy method and (\ref{1.18}), (\ref{1.19}) we derive inequality (\ref{3.51}) of the form
\begin{equation}
X\le\phi(D_1(X),D_2(X),X/\nu,X/\sqrt{\nu},X(0)),
\label{1.20}
\end{equation}
where $\phi$ is an increasing positive function of polynomial type. In (\ref{1.20}) it is assumed that the quantity
\begin{equation}
X(0)=\sqrt{\nu}\|\nabla\varphi(0)\|_{\Gamma_1^2(\Omega)}+ \|\rot\psi(0)\|_{\Gamma_1^2(\Omega)}
\label{1.21}
\end{equation}
is finite. This implies the foolowing smallness condition
\begin{equation}
\|\nabla\varphi(0)\|_{\Gamma_1^2(\Omega)}\le c/\sqrt{\nu}.
\label{1.22}
\end{equation}
However, to show (\ref{1.18}) we needed that
\begin{equation}
c_1/\nu^\varkappa\le|\varphi(0)|_p\le c_2/\nu^\varkappa,\quad \varkappa\in(1/2,1),\quad p\in(3,6).
\label{1.23}
\end{equation}
We have to emphasize that the energy method does not work without (\ref{1.18}), (\ref{1.19}), because estimates can not be closed.

Since $\nu$ is finite estimate (\ref{1.12}) derived from (\ref{1.20}) implies long time existence of solutions to problem (\ref{1.7}). To prove global existence we derive (\ref{1.14}), where $B$ is small for large $\nu$.

\noindent
This implies that initial data for problem (\ref{1.7}) at time $t=T$ are small. Hence the small data technique implies global existence of regular solutions to (\ref{1.7}) in $(T,\infty)$.

We have to emphasize that we are restricted to derive appropriate estimates only. Then the local solution can be proved by the method of successive approximations and the continuation argument implies global existence.
\end{remark}

\noindent
Theorem \ref{t1.3} follows from Lemmas \ref{l5.1}--\ref{l5.4} by the stability argument which works for $\nu$ sufficiently large.

\noindent
Theorem \ref{t1.4} follows directly from Theorems \ref{t1.2} and \ref{t1.3}.

\noindent
We use the simplified notation

\begin{notation}\label{n1.6}
$$\eqal{
&\|u\|_{L_p(\Omega)}=|u|_p,\quad \|u\|_{H^s(\Omega)}=\|u\|_s,\quad \|u\|_{W_p^s(\Omega)}=\|u\|_{s,p},\cr
&H^s(\Omega)=W_2^s(\Omega),\quad \|u\|_{W_p^s(\Omega)}=\bigg(\sum_{|\alpha|\le s}\intop_\Omega|D_x^\alpha u(x)|^pdx\bigg)^{1/p},\cr
&D_x^\alpha=\partial_{x_1}^{\alpha_1}\partial_{x_2}^{\alpha_2}\partial_{x_3}^{\alpha_3},\quad |\alpha|=\alpha_1+\alpha_2+\alpha_3,\qquad s\in\N,\quad p\in[1,\infty],\cr
&|u|_{k,l}^2=\sum_{i=0}^l\|\partial_t^iu\|_{k-i}^2,\quad |u|_{k,l,r,\Omega^t}=\bigg(\intop_0^t|u(t')|_{k,l}^rdt'\bigg)^{1/r},\cr
&|u|_{r,q,\Omega^t}=\bigg(\intop_0^t|u(t')|_r^qdt'\bigg)^{1/q},\quad  r,q\in[1,\infty],\cr
&\eqal{&\|u\|_{L_r(0,t;H^s(\Omega))}=\|u\|_{s,r,\Omega^t},\cr
&\|u\|_{L_r(0,t;W_p^s(\Omega))}=\|u\|_{s,p,r,\Omega^t},\cr}\quad\ \  s\in\N,\quad r,p\in[1,\infty].\cr}
$$
Introduce the space
$$
\Gamma_l^k(\Omega)=\{u\colon|u|_{k,l}<\infty\},\quad l\le k,\quad l,k\in\N_0=\N\cup\{0\}.
$$
For $t_0$ fixed we have
$$\|u(t_0)\|_{\Gamma_l^k(\Omega)}=\|u(t_0)\|_{H^k(\Omega)}+\sum_{i=1}^l\|\partial_t^iu(t)|_{t=t_0}\|_{H^{k-i}(\Omega)}.
$$
By $\phi$, $\phi_\sigma$, $\sigma\in\N$, we always denote increasing positive continuous functions of their argumens.
\end{notation}

\noindent
Finally, we introduce some notation

\begin{notation}\label{n1.7}
$$\eqal{
\Psi(t)&=\nu|\nabla\varphi|_{3,1,2,\Omega^t}\quad \chi_0=\sqrt{\nu}|\nabla\varphi|_{2,1,\infty,\Omega^t},\cr
X^2(t)&=\nu|\nabla\varphi(t)|_{2,1}^2+|\rot\psi(t)|_{2,1}^2+ \mu(\nu|\nabla\varphi|_{3,1,2,\Omega^t}^2\cr
&\quad+|\rot\psi|_{3,1,2,\Omega^t}^2)+\nu^2|\nabla\varphi|_{3,1,2,\Omega^t}^2,\cr
X^2(0)&=\nu|\nabla\varphi(0)|_{2,1}^2+|\rot\psi(0)|_{2,1}^2,\cr
Y^2(t)&=\nu|\nabla\varphi(t)|_{2,1}^2+|\rot\psi(t)|_{2,1}^2.\cr}
$$
Let $\Omega_T^t=\Omega\times(T,t)$, $t>T$. Let $\Psi_T$, $\chi_{0T}$, $X_T$ contain $\Omega_T^t$ instead of $\Omega^t$. Then
$$
X_T^2(T)=\nu|\nabla\varphi(T)|_{2,1}^2+|\rot\psi(T)|_{2,1}^2.
$$
$$\eqal{
&A_1=|v(0)|_2,\quad\! D_1=c_1\Psi^{2/3}\phi_1(\Psi^{1/3}/\nu^{1/3}+1/\nu^{(\varkappa-1/2)/3})+ c_1(|v(0)|_6+A_1),\cr
&\phi_1=\exp\bigg[{c|v|_{2p/(p-2),2/(1-\varkappa),\Omega^t}^{2/(1-\varkappa)} \over [(\mu+\nu)^\varkappa\min_t(|\varphi(0)|_p-{\sqrt{t}\over\nu}\Psi)]^{1/(1-\varkappa)}}\bigg],
\quad p\in(3,6),\cr
&\varkappa=3/2-3/p\in(1/2,1),\quad D_2=|v_t(0)|_2\exp(D_1^2A_1^2).\cr}
$$
\end{notation}

\noindent
There is a huge literature concerning the regularity problem of weak solutions to the Navier-Stokes equations. Therefore we are not able to present all papers devoted to this problem. Moreover, we are not able to close the list of mathematicians trying to solve it. Hence, we concentrate the presentation on some directions and recall mathematicians working in these areas.
\vskip6pt

\begin{itemize}
\item[1.] Formulation of sufficient conditions guaranteeing regularity of weak solutions\\
The first who formulated such conditions was J. Serrin \cite{S}. This approach was continued by D. Chae, H.J. Choe, H. Kozono, H. Sohr, J. Neustupa, M. Pokorny, P. Penel and the references of their papers are cited in \cite{Z8, Z9}. We have to recall results of G. Seregin, V. \v Sver\'ak, T. Shilkin, A. Mikhaylov (see \cite{S1, S2, S3, S4, SSS, SS1, ESS, MS}).
\item[2.] Local regularity theory.\\
The direction of examining regularity of weak solutions of the Navier-Stokes equations was initiated in the celebrated paper of L. Caffarelli, R. Kohn, L. Nirenberg (see \cite{CKN}). The famous mathematicians working in this directions are G. Seregin \cite{S1, S2, S3, S4}, V. \v Sver\'ak \cite{SS1}.
\item[3.] Rotating Navier-Stokes equations.\\
The existence of global regular solutions to the rotating Navier-Stokes equations was strongly examined by A. Babin, A. Mahalov, B. Nicolaenko (see \cite{BMN1, BMN2, BMN3, MN}).
\item[4.] Global regular solutions to the Navier-Stokes equations with some special properties. We have to distinguish the following directions
\begin{itemize}
\item[a.] Thin domains (see \cite{RS1, RS2, RS3}).
\item[b.] Small variations of vorticity (see \cite{CF}).
\item[c.] Motions in cylindrical domains with data close to data of 2d solutions (see \cite{Z10, Z11, Z12, Z13, NZ}).
\item[d.] Motions in axisymmetric domains with data close to data of axisymmetric solutions (see \cite{Z3, Z4}).
\end{itemize}
\end{itemize}

\section{Notation and auxiliary results}\label{s2}

First we obtain the energy type estimate for solutions to (\ref{1.2}).

\begin{lemma}\label{l2.1}
Let $v$ be a solution to (\ref{1.7}). Assume that $v(0)\in L_2(\Omega)$.\\
Then the following estimate holds
\begin{equation}
|v(t)|_2^2+\mu\|v\|_{1,2,\Omega^t}^2+\nu|\Delta\varphi|_{2,\Omega^t}^2\le c|v(0)|_2^2\equiv A_1^2.
\label{2.1}
\end{equation}
\end{lemma}

\begin{proof}
Multiply (\ref{1.7}) by $v$, integrate over $\Omega$, integrate by parts, exploit (\ref{1.4}), (\ref{1.6}) and integrate with respect to time. Then we obtain (\ref{2.1}) and conclude the proof.
\end{proof}

\begin{lemma}\label{l2.2}
Let $v$ be a solution to (\ref{1.7}). Assume that $v(0)\in L_r(\Omega)$, $r>2$. Let $\nu|\divv v|_{3r/(r+1),r,\Omega^t}$ be finite.\\
Then there exists a constant $c_1=c_1(r,\mu,c_0)$, where $c_0$ is the constant from imbedding (\ref{2.7}), such that
\begin{equation}\eqal{
&|v(t)|_r+|v|_{3r,r,\Omega^t}+|\nabla|v|^{r/2}|_{2,\Omega^t}^{2/r}\cr
&\le c_1(r,\mu,c_0)[\nu|\divv v|_{3r/(r+1),r,\Omega^t}+|v(0)|_r+A_1]\cr}
\label{2.2}
\end{equation}
where $c_1=\max\{\bar c,1\}$ and $\bar c$ is introduced in (\ref{2.9}).
\end{lemma}

\begin{proof}
Multiplying (\ref{1.7}) by $v|v|^{r-2}$ and integrating over $\Omega$ yield
\begin{equation}
{1\over r}{d\over dt}|v|_r^r+\mu\intop_\Omega\nabla v\cdot\nabla(v|v|^{r-2})dx+\nu\intop_\Omega\divv v\divv(v|v|^{r-2})dx=0.
\label{2.3}
\end{equation}
First we consider
$$\eqal{
J_1&=\mu\intop_\Omega\nabla v\cdot\nabla(v|v|^{r-2})dx=\mu\intop_\Omega|\nabla v|^2|v|^{r-2}dx\cr
&\quad+\mu\intop_\Omega v_k\nabla v_k(r-2)|v|^{r-3}\nabla|v|dx\equiv I_1+I_2.\cr}
$$
Using that $v_k\nabla v_k={1\over2}\nabla|v|^2=|v|\nabla|v|$ we have
$$\eqal{
I_2&=(r-2)\mu\intop_\Omega|v|^{r-2}|\nabla|v|\,|^2dx=(r-2)\mu\intop_\Omega|\,|v|^{{r\over2}-1}\nabla|v|\,|^2dx\cr
&={4(r-2)\mu\over r^2}\intop_\Omega|\nabla|v|^{r/2}|^2dx.\cr}
$$
To examine $I_1$ we use the formula
\begin{equation}
|\nabla u|^2=|u|^2\bigg|\nabla{u\over|u|}\bigg|^2+|\nabla|u|\,|^2.
\label{2.4}
\end{equation}
Then $I_1$ takes the form
$$\eqal{
I_1&=\mu\intop_\Omega\bigg[|v|^2\bigg|\nabla{v\over|v|}\bigg|^2+|\nabla|v|\,|^2\bigg] |v|^{r-2}dx\cr
&=\mu\intop_\Omega|v|^r\bigg|\nabla{v\over|v|}\bigg|^2dx+{4\mu\over r^2}\intop_\Omega|\nabla|v|^{r/2}|^2dx.\cr}
$$
Hence,
$$
J_1={4(r-1)\over r^2}\mu\intop_\Omega|\nabla|v|^{r/2}|^2dx+\mu\intop_\Omega|v|^r \bigg|\nabla{v\over|v|}\bigg|^2dx.
$$
Next, we consider
$$\eqal{
J_2&=\nu\intop_\Omega\divv v\divv(v|v|^{r-2})dx=\nu\intop_\Omega|\divv v|^2|v|^{r-2}dx\cr
&\quad+\nu\intop_\Omega\divv vv\cdot\nabla|v|^{r-2}dx\equiv I_3+I_4,\cr}
$$
where
$$
|I_4|\le\nu(r-2)\intop_\Omega|\divv v|\,|v|^{r-2}|\nabla|v|\,|dx.
$$
Employing the above expressions in (\ref{2.3}) one gets
\begin{equation}\eqal{
&{1\over r}{d\over dt}|v|_r^r+{4(r-1)\mu\over r^2}\intop_\Omega|\nabla|v|^{r/2}|^2dx+\mu\intop_\Omega|v|^r \bigg|\nabla{v\over|v|}\bigg|^2dx\cr
&\quad+\nu\intop_\Omega|\divv v|^2|v|^{r-2}dx\le\nu(r-2)\intop_\Omega|\divv v|\,|v|^{r-2}|\nabla|v|\,|dx\cr
&=\nu(r-2)\intop_\Omega|\divv v|\,|v|^{{r\over2}-1}|v|^{{r\over2}-1}|\nabla|v|\,|dx\cr
&\le{\varepsilon\over2}\intop_\Omega|\,|v|^{{r\over2}-1}\nabla|v|\,|^2dx+ {\nu^2(r-2)^2\over2\varepsilon}\intop_\Omega|\divv v|^2|v|^{r-2}dx\cr
&={2\varepsilon\over r^2}\intop_\Omega|\nabla|v|^{r/2}|^2dx+{\nu^2(r-2)^2\over2\varepsilon}\intop_\Omega |\divv v|^2|v|^{r-2}dx.\cr}
\label{2.5}
\end{equation}
Setting $\varepsilon=(r-1)\mu$ we obtain from (\ref{2.5}) the inequality
\begin{equation}\eqal{
&{1\over r}{d\over dt}|v|_r^r+{2(r-1)\mu\over r^2}\intop_\Omega|\nabla|v|^{r/2}|^2dx+\mu\intop_\Omega|v|^r \bigg|\nabla{v\over|v|}\bigg|^2dx\cr
&\quad+\nu\intop_\Omega|\divv v|^2|v|^{r-2}dx\le{\nu^2(r-2)^2\over2(r-1)\mu}\intop_\Omega|\divv v|^2|v|^{r-2}dx.\cr}
\label{2.6}
\end{equation}
Consider the second integral on the l.h.s. of (\ref{2.6}). Omitting the coefficient we write it in the form
$$
I\equiv\intop_\Omega|\nabla|v|^{r/2}|^2dx={1\over3}|\nabla|v|^{r/2}|_2^2+{1\over3} |\nabla|v|^{r/2}|_2^2+{1\over3}|\nabla|v|^{r/2}|_2^2.
$$
Let $u=|v|^{r/2}$. Then we use the Poincar\'e inequality
$$
\intop_\Omega\bigg|u-\diagintop u\bigg|^2dx\le c_p|\nabla u|_2^2.
$$
Hence
$$
|u|_2^2\le c_p|\nabla u|_2^2+\bigg|\diagintop_\Omega udx\bigg|^2.
$$
Recalling that $u=|v|^{r/2}$ we have
$$
{1\over3}|\nabla|v|^{r/2}|_2^2\ge{1\over3c_p}\bigg(|v|_r^r- \bigg|\diagintop_\Omega|v|^{r/2}dx\bigg|^2\bigg).
$$
Therefore
$$
I\ge{1\over6c_p}|v|_r^r+{1\over6c_p}|v|_r^r+{1\over3}|\nabla|v|^{r/2}|_2^2+{1\over3} |\nabla|v|^{r/2}|_2^2-{1\over3c_p}\bigg|\diagintop|v|^{r/2}dx\bigg|^2.
$$
Let $\bar c_1=\min\big({1\over2c_P},1\big)$. Next we use the inequality
\begin{equation}
\bar c_0|v|_{3r}^r\le|\nabla|v|^{r/2}|_2^2+|v|_r^r.
\label{2.7}
\end{equation}
Then we have
$$
I\ge{1\over6c_p}|v|_r^r+{\bar c_1\bar c_0\over3}|v|_{3r}^2+{1\over3}|\nabla|v|^{r/2}|_2^2-{1\over3c_p} \bigg|\diagintop|v|^{r/2}dx\bigg|^2.
$$
Finally, we estimate the expression
$$\eqal{
&\bigg|\intop_\Omega|v|^{r/2}dx\bigg|^2=\bigg|\intop_\Omega|v|^{r/2-\mu}|v|^\mu dx\bigg|^2\le\intop_\Omega|v|^{r-2\mu}dx\intop_\Omega|v|^{2\mu}dx\cr
&\le|\Omega|^{2\mu\over r}\bigg(\intop_\Omega|v|^rdx\bigg)^{(r-2\mu)/r}\intop_\Omega|v|^{2\mu}dx\equiv J,\cr}
$$
where $|\Omega|=$ volume of $\Omega$. Setting $\mu=1$ we have
$$
J=|\Omega|^{2\over r}|v|_r^{r-2}|v|_2^2\le\varepsilon_3|v|_r^r+c/\varepsilon_3 |\Omega|\,|v|_2^r.
$$
Using the estimates in the lower bound for $I$ we obtain for sufficiently small $\varepsilon_3$ the inequality
$$\eqal{
I&\ge{1\over 10c_p}|v|_r^r+ {\bar c_1\bar c_0\over3}|v|_{3r}^r+{1\over3}|\nabla|v|^{r/2}|_2^2- c|v|_2^r\cr
&\equiv c_1|v|_r^r+c_0|v|_{3r}^r+{1\over3}|\nabla|v|^{r/2}|_2^2-c|v|_2^r.\cr}
$$
Then (\ref{2.6}) takes the form
\begin{equation}\eqal{
&{1\over r}{d\over dt}|v|_r^r+{(r-1)\mu\over r^2}c_1|v|_r^r+{(r-1)\mu c_0\over r^2}|v|_{3r}^r+{(r-1)\over3r^2}\intop_\Omega|\nabla|v|^{r/2}|^2dx\cr
&\quad+\mu\intop_\Omega|v|^r\bigg|\nabla{v\over|v|}\bigg|^2dx+\nu\intop_\Omega |\divv v|^2|v|^{r-2}dx\cr
&\le{\nu^2(r-2)^2\over2(r-1)\mu}|\divv v|_{3r/(r+1)}^2|v|_{3r}^{r-2}+c|v|_2^r\cr
&\le{\varepsilon_1^{r/(r-2)}\over r/(r-2)}|v|_{3r}^r+{1\over\varepsilon_1^{r/2}r/2} \bigg({\nu^2(r-2)^2\over2(r-1)\mu}\bigg)^{r/2}|\divv v|_{3r/(r+1)}^r+c|v|_2^r.\cr}
\label{2.8}
\end{equation}
Setting
$$
{\varepsilon_1^{r/(r-2)}\over r/(r-2)}={(r-1)\mu c_0\over 2r^2}
$$
we have that
$$
\varepsilon_1=\bigg({(r-1)\mu c_0\over2(r-2)r}\bigg)^{(r-2)/r}.
$$
Then the coefficient near $|\divv v|_{3r/(r+1)}^r$ from the r.h.s. of (\ref{2.8}) equals
\begin{equation}
(r-2)\bigg[{r-2\over(r-1)\mu}\bigg]^{r-1}\cdot{r^{r/2-2}\over c_0^{r/2-1}}\nu^r\equiv \bar c(r,\mu,c_0)\nu^r.
\label{2.9}
\end{equation}
Using the above expressions in the r.h.s. of (\ref{2.8}) and integrating the result with respect to time we obtain
\begin{equation}\eqal{
&{1\over r}|v(t)|_r^r+{(r-1)\mu c_0\over 2r^2}|v|_{3r,r,\Omega^t}^r+ {(r-1)\mu\over 3r^2}|\nabla|v|^{r/2}|_{2,\Omega^t}^2\cr
&\quad+\mu\bigg|\,|v|^{r/2}\bigg|\nabla{v\over|v|}\bigg|\,\bigg|_{2,\Omega^t}^2+ \nu\bigg|\,|\divv v|\,|v|^{{r\over2}-1}\bigg|_{2,\Omega^t}^2\cr
&\le\bar c\nu^r|\divv v|_{3r/(r+1),r,\Omega^t}^r+{1\over r}|v(0)|_r^r+A_1^r,\cr}
\label{2.10}
\end{equation}
where we used that
$$
\intop_0^t|v|_2^rdt'\le|v|_{2,\infty,\Omega^t}^{r-2}|v|_{2,\Omega^t}^2\le A_1^r.
$$
Inequality (\ref{2.10}) implies (\ref{2.2}). This concludes the proof.
\end{proof}

\begin{remark}\label{r2.3}
The first term on the r.h.s. of (\ref{2.2}) equals $|\Delta\varphi|_{3r/(r+1),r,\Omega^t}$. To examine it we first consider the interpolation
\begin{equation}
|\Delta\varphi|_{3r/(r+1)}\le c|\nabla\Delta\varphi|_2^\theta|\nabla\varphi|_2^{1-\theta},
\label{2.11}
\end{equation}
where $\theta=3/4-1/2r$, $1-\theta=1/4+1/2r$. Therefore, we have
\begin{equation}
|\Delta\varphi|_{3r/(r+1),r,\Omega^t}\le c|\nabla\Delta\varphi|_{2,\infty,\Omega^t}^{3/4-1/2r}\bigg(\intop_0^t |\nabla\varphi(t')|_2^{(1/4+1/2r)r}dt'\bigg)^{1/r},
\label{2.12}
\end{equation}
where $(1/4+1/2r)r\le 2$ for $r\le 6$. For $r=6$ (\ref{2.12}) takes the form
\begin{equation}
|\Delta\varphi|_{18/7,6,\Omega^t}\le c|\nabla\Delta\varphi|_{2,\infty,\Omega^t}^{2/3}|\nabla\varphi|_{2,\Omega^t}^{1/3}.
\label{2.13}
\end{equation}
Introduce the quantity
\begin{equation}
\Psi=\nu|\nabla\varphi|_{3,1,2,\Omega^t}.
\label{2.14}
\end{equation}
Using the imbeddings
$$\eqal{
&\|\Delta\varphi\|_{2,2,\Omega^t}\le\|\nabla\varphi\|_{3,2,\Omega^t}\le |\nabla\varphi|_{3,1,2,\Omega^t},\cr
&|\Delta\varphi_t|_{2,\Omega^t}\le\|\Delta\varphi_t\|_{1,2,\Omega^t}\le |\nabla\varphi|_{3,1,2,\Omega^t},\cr
&|\nabla\Delta\varphi|_{2,\infty,\Omega^t}\le c\|\Delta\varphi\|_{W_2^{2,1}(\Omega^t)}\le c|\nabla\varphi|_{3,1,2,\Omega^t}\cr}
$$
we obtain from (\ref{2.13}) the inequality
\begin{equation}
\nu|\Delta\varphi|_{18/7,6,\Omega^t}\le c\Psi^{2/3}\nu^{1/3}|\nabla\varphi|_{2,\Omega^t}^{1/3}\equiv I.
\label{2.15}
\end{equation}
Our aim is to find the estimate
\begin{equation}
I\le c\sum_{\alpha,\beta}{\Psi^\alpha\over\nu^\beta},
\label{2.16}
\end{equation}
where $\alpha$, $\beta$ are positive constants. This ends the Remark.
\end{remark}

To show (\ref{2.16}) we derive from (\ref{1.2}) the equation
\begin{equation}
\Delta\varphi_t-(\mu+\nu)\Delta^2\varphi=-\divv(\rot\psi\cdot\nabla v).
\label{2.17}
\end{equation}
Applying the operator $\Delta^{-1}$ we obtain
\begin{equation}\eqal{
&\varphi_t-(\mu+\nu)\Delta\varphi=-\Delta^{-1}\partial_{x_i}\partial_{x_j}(v_iv_j) +\Delta^{-1}\partial_{x_i}\partial_{x_j}(\varphi_{x_i}v_j),\cr
&\varphi|_{t=0}=\varphi(0).\cr}
\label{2.18}
\end{equation}

\begin{lemma}\label{l2.4}
Assume that $v\in L_{2p/(p-1),2/(1-\varkappa)}(\Omega^t)$, $\varkappa=3/2-3/p\in(1/2,1)$, $p\in(3,6)$, $v(0)\in L_2(\Omega)$, $\varphi(0)\in L_p(\Omega)$. Assume that there exist positive constants $c_2$, $c_3$, $c_4$ such that
$$
(\mu+\nu)^\varkappa\bigg(|\varphi(0)|_p-{\sqrt{t}\over\nu}\Psi\bigg)\ge c_2>0,\quad {c_3\over\nu^\varkappa}\le|\varphi(0)|_p\le{c_4\over\nu^\varkappa}.
$$
Then
\begin{equation}\eqal{
&|\varphi(t)|_2^2+(\mu+\nu)|\nabla\varphi|_{2,\Omega^t}^2\cr
&\le\exp\bigg[ (c|v|_{2p/(p-1),2/(1-\varkappa),\Omega^t}^{2/(1-\varkappa)})/ \bigg[(\mu+\nu)^\varkappa\min_t\bigg(|\varphi(0)|_p- {\sqrt{t}\over\nu}\Psi\bigg)\bigg]^{1/(1-\varkappa)}\bigg]\cdot\cr
&\quad\cdot\bigg[c{\Psi^2\over\nu^3}+{c_4\over\nu^{2\varkappa}}\bigg]\equiv\phi_1 \cdot\bigg(c{\Psi^2\over\nu^3}+{c_4\over\nu^{2\varkappa}}\bigg)\cr}
\label{2.19}
\end{equation}
and
$$
|\nabla\varphi|_{2,\Omega^t}^2\le\phi_1\cdot \bigg(c{\Psi^2\over\nu^4}+{c_4\over\nu^{2\varkappa+1}}\bigg)
$$
where $2\varkappa+1>2$ because $\varkappa>1/2$ and
$$
\phi_1=\exp\bigg[{c|v|_{2p/(p-1),2/(1-\varkappa),\Omega^t}^{2/(1-\varkappa)}\over [(\mu+\nu)^\varkappa\min_t(|\varphi(0)|_p-{\sqrt{t}\over\nu}\Psi)]^{1/(1-\varkappa)}}\bigg].
$$
\end{lemma}

\begin{proof}
Multiplying (\ref{2.18}) by $\varphi$ and integrating over $\Omega$ yield
\begin{equation}\eqal{
&{1\over2}{d\over dt}|\varphi|_2^2+(\mu+\nu)|\nabla\varphi|_2^2=\intop_\Omega (-\Delta^{-1}\partial_{x_i}\partial_{x_j}(v_iv_j)\varphi) dx\cr
&\quad+\intop_\Omega\Delta^{-1}\partial_{x_i}\partial_{x_j}(\varphi_{x_i}v_j)\varphi dx\equiv\intop_\Omega\bar D_1\varphi dx+\intop_\Omega\bar D_2\varphi dx\cr
&\equiv I_1+I_2.\cr}
\label{2.20}
\end{equation}
First we examine $I_1$. By the H\"older inequality we have
$$
|I_1|\le|\bar D_1|_{p/(p-1)}|\varphi|_p={|\bar D_1|_{p/(p-1)}|\varphi|_p^2\over |\varphi|_p}\equiv I_1,
$$
where $1<p<6$. Let $\alpha={|\bar D_1|_{p/(p-1)}\over|\varphi|_p}$. Then we use the interpolation
$$
I_1^{1/2}=\alpha^{1/2}|\varphi|_p\le\alpha^{1/2}(\varepsilon^{1/\varkappa} |\nabla\varphi|_2+c\varepsilon^{-1/(1-\varkappa)}|\varphi|_2),
$$
where $\varkappa=3/2-3/p$, $p\in(3,6)$. Setting $\varepsilon^{1/\varkappa}\alpha^{1/2}=\big({\mu+\nu\over k}\big)^{1/2}$, $k\in\N$, we have $\varepsilon=\big({\mu+\nu\over\alpha k}\big)^{\varkappa/2}$. Then
$$
\alpha^{1/2}\varepsilon^{-1/(1-\varkappa)}={(\alpha k^\varkappa)^{1/2(1-\varkappa)}\over(\mu+\nu)^{\varkappa/2(1-\varkappa)}}.
$$
Therefore,
\begin{equation}\eqal{
I_1&\le{1\over k}(\mu+\nu)|\nabla\varphi|_2^2+{c(k^\varkappa\alpha)^{1/(1-\varkappa)}\over (\mu+\nu)^{\varkappa/(1-\varkappa)}}|\varphi|_2^2\cr
&={1\over k}(\mu+\nu)|\nabla\varphi|_2^2+{ck^{\varkappa/(1-\varkappa)}|\bar D_1|_{p/(p-1)}^{1/(1-\varkappa)}\over[(\mu+\nu)^\varkappa|\varphi|_p]^{1/(1-\varkappa)}} |\varphi|_2^2,\cr}
\label{2.21}
\end{equation}
where $|\bar D_1|_q\le c\sum_{i,j=1}^3|v_iv_j|_q$ for any $q\in(1,\infty)$.

\noindent
By the H\"older and Young inequalities we get
\begin{equation}\eqal{
|I_2|&=\bigg|\intop_\Omega\Delta^{-1}\partial_{x_j}(\varphi_{x_i}v_j) \varphi_{x_i}dx\bigg|\cr
&\le{\mu+\nu\over k}|\nabla\varphi|_2^2+{ck\over\mu+\nu}\sum_{i,j=1}^3 |\Delta^{-1}\partial_{x_j}(\varphi_{x_i}v_j)|_2^2\cr
&\le{\mu+\nu\over k}|\nabla\varphi|_2^2+{ck\over\mu+\nu}\sum_{i,j=1}^3\|\Delta^{-1}\partial_{x_j} (\varphi_{x_i}v_j)\|_{1,6/5}^2\cr
&\le{\mu+\nu\over k}|\nabla\varphi|_2^2+{ck\over\mu+\nu}\sum_{i,j=1}^3(|\nabla\Delta^{-1}\partial_{x_j} (\varphi_{x_i}v_j)|_{6/5}^2+|\Delta^{-1}\partial_{x_j}(\varphi_{x_i}v_j)|_{6/5}^2)\cr
&\le{\mu+\nu\over k}|\nabla\varphi|_2^2+{ck\over\mu+\nu}\sum_{i,j=1}^3|\nabla\Delta^{-1}\partial_{x_j} (\varphi_{x_i}v_j)|_{6/5}^2\cr
&\le{\mu+\nu\over k}|\nabla\varphi|_2^2+{ck\over\mu+\nu}\sum_{i,j=1}^3 |\varphi_{x_i}v_j|_{6/5}^2\cr
&\le{\mu+\nu\over k}|\nabla\varphi|_2^2+{ck\over\mu+\nu}|\varphi_x|_3^2|v|_2^2,\cr}
\label{2.22}
\end{equation}
where in the fourth inequality the Poincar\'e inequality is used.

\noindent
Using (\ref{2.21}) and (\ref{2.22}) in (\ref{2.20}), assuming that $k=4$ and integrating the result with respect to time yield
\begin{equation}\eqal{
&|\varphi(t)|_2^2+(\mu+\nu)|\nabla\varphi|_{2,\Omega^t}^2\le\exp\bigg( {|v|_{2p/(p-1),2/(1-\varkappa),\Omega^t}^{2/(1-\varkappa)}\over [(\mu+\nu)^\varkappa\min_t|\varphi(t)|_p]^{1/(1-\varkappa)}}\bigg)\cdot\cr
&\quad\cdot\bigg[{c\over\mu+\nu}|\varphi_x|_{3,2,\Omega^t}^2|v|_{2,\infty,\Omega^t}^2+ |\varphi(0)|_2^2\bigg].\cr}
\label{2.23}
\end{equation}
From Lemma \ref{l2.1} we have that $|v|_{2,\infty,\Omega^t}\le A_1$. Moreover $|\varphi_x|_{3,2,\Omega^t}\le{\Psi\over\nu}$. To guarantee that the argument of exp is finite we consider
$$
\varphi(t)=\varphi(0)+\intop_0^t\varphi_{t'}dt'
$$
so
$$\eqal{
|\varphi(t)|_p&\ge|\varphi(0)|_p-\bigg|\intop_0^t\varphi_{t'}dt'\bigg|_p\ge |\varphi(0)|_p-\intop_0^t|\varphi_{t'}|_pdt'\cr
&\ge|\varphi(0)|_p-{\sqrt{t}\over\nu}\Psi.\cr}
$$
To have the argument of exp finite we assume existence of a positive constant $c_2$ such that
\begin{equation}
(\mu+\nu)^\varkappa\bigg(|\varphi(0)|_p-{\sqrt{t}\over\nu}\Psi\bigg)\ge c_2.
\label{2.24}
\end{equation}
Such constant exists at least for the local solution. Moreover, we need
\begin{equation}
{c_3\over\nu^\varkappa}\le|\varphi(0)|_p\le{c_4\over\nu^\varkappa},
\label{2.25}
\end{equation}
where $c_3$, $c_4$ are positive constants. The above considerations imply (\ref{2.19}) and conclude the proof.
\end{proof}

\begin{remark}\label{r2.5}
Using (\ref{2.19}) in (\ref{2.15}) yields
\begin{equation}
\nu|\Delta\varphi|_{18/7,6,\Omega^t}\le c\Psi^{2/3}\phi_1\cdot\bigg({\Psi^{1/3}\over\nu^{1/3}}+ {1\over\nu^{(\varkappa-1/2)/3}}\bigg).
\label{2.26}
\end{equation}
Using (\ref{2.26}) in (\ref{2.2}) gives
\begin{equation}\eqal{
&|v(t)|_6+|v|_{18,6,\Omega^t}+|\nabla|v|^3|_{2,\Omega^t}^{1/3}\cr
&\le c_1\Psi^{2/3}\phi_1\cdot(\Psi^{1/3}/\nu^{1/3}+1/\nu^{(\varkappa-1/2)/3})+ c_1(|v(0)|_6+A_1)\cr
&\equiv D_1,\cr}
\label{2.27}
\end{equation}
where
$$
\phi_1=\exp\bigg[{c|v|_{2p/(p-1),2/(1-\varkappa),\Omega^t}^{2/(1-\varkappa)}\over [(\mu+\nu)^\varkappa\min_t(|\varphi(0)|_p-{\sqrt{t}\over\nu}\Psi)]^{1/(1-\varkappa)}}\bigg]
$$
and $c_1$ is defined in (\ref{2.2}).
\end{remark}

\begin{lemma}\label{l2.6}
Assume that $v_t(0)\in L_2(\Omega)$ and the assumptions of Lemmas \ref{l2.1} and \ref{l2.4} hold. Let $D_1$, defined by (\ref{2.27}), be finite.\\
Then
\begin{equation}\eqal{
&|v_t(t)|_2^2+\mu\|v_t\|_{1,2,\Omega^t}^2+\nu|\Delta\varphi_t|_{2,\Omega^t}^2\cr
&\le|v_t(0)|_2^2\exp(D_1^2A_1^2)\equiv D_2^2.\cr}
\label{2.28}
\end{equation}
\end{lemma}

\begin{proof}
Differentiate (\ref{1.7}) with respect to $t$, multiply by $v_t$ and integrate over $\Omega$. Then we have
\begin{equation}\eqal{
&{1\over2}{d\over dt}|v_t|_2^2+\mu|\nabla v_t|_2^2+\nu|\divv v_t|_2^2=-\intop_\Omega\rot\psi\cdot\nabla v_t\cdot v_tdx\cr
&\quad-\intop_\Omega\rot\psi_t\cdot\nabla v\cdot v_tdx\equiv J_1+J_2.\cr}
\label{2.29}
\end{equation}
Integration by parts implies that $J_1=0$. Next
$$
|J_2|=\bigg|\intop_\Omega\rot\psi_t\cdot\nabla v_t\cdot vdx\bigg|\le\varepsilon|\nabla v_t|_2^2+c/\varepsilon|\rot\psi_t|_3^2|v|_6^2.
$$
Using the estimate in (\ref{2.29}) yields
\begin{equation}\eqal{
&{d\over dt}|v_t|_2^2+\mu\|v_t\|_1^2+\nu|\Delta\varphi_t|_2^2\le c|\rot\psi_t|_3^2|v|_6^2\cr
&\le c|\nabla\rot\psi_t|_2|\rot\psi_t|_2|v|_6^2\le\varepsilon|\nabla v_t|_2^2+c/\varepsilon|\rot\psi_t|_2^2|v|_6^4\cr}
\label{2.30}
\end{equation}
Continuing, we have
\begin{equation}
{d\over dt}|v_t|_2^2+\mu\|v_t\|_1^2+\nu|\Delta\varphi_t|_2^2\le c|v_t|_2^2|v|_6^4.
\label{2.31}
\end{equation}
Integration with respect to time yields
$$\eqal{
&|v_t(t)|_2^2+\exp\bigg(c\intop_0^t|v(t')|_6^4dt'\bigg)\intop_0^t (\mu\|v_{t'}(t')\|_1^2+\nu|\Delta\varphi_{t'}(t')|_2^2)\cdot\cr
&\quad\cdot\exp\bigg(-c\intop_0^{t'}|v(t'')|_6^4dt''\bigg)dt'\le|v_t(0)|_2^2 \exp \bigg(c\intop_0^t|v(t')|_6^4dt'\bigg).\cr}
$$
Simplifying we get (\ref{2.28}). This concludes the proof.
\end{proof}

\section{Estimates and existence}\label{s3}

First we derive estimates for solutions to (\ref{1.7}) by applying the energy method.

\begin{lemma}\label{l3.1}
Assume that $A_1$, $D_1$ are finite, $\nabla\varphi(0)$, $\rot\psi(0)\in L_2(\Omega)$.\\
Then
\begin{equation}\eqal{
&|\nabla\varphi(t)|_2^2+{1\over\nu}|\rot\psi(t)|_2^2+ \mu\bigg( |\nabla^2\varphi|_{2,\Omega^t}^2+{1\over\nu}|\nabla\rot\psi|_{2,\Omega^t}^2\bigg)\cr
&\quad+\nu|\Delta\varphi|_{2,\Omega^t}^2\le{c\over\nu}A_1^2D_1^4+ |\nabla\varphi(0)|_2^2+{1\over\nu}|\rot\psi(0)|_2^2.\cr}
\label{3.1}
\end{equation}
\end{lemma}

\begin{proof}
Multiply (\ref{1.7}) by $\nabla\varphi$ and integrate over $\Omega$. Then we have
\begin{equation}
{1\over2}{d\over dt}|\nabla\varphi|_2^2+\intop_\Omega\rot\psi\cdot\nabla v\cdot\nabla\varphi dx+\mu|\nabla^2\varphi|_2^2+\nu|\Delta\varphi|_2^2=0.
\label{3.2}
\end{equation}
Integration by parts in the second term yields
$$
\bigg|\intop_\Omega\rot\psi\cdot v\cdot\nabla^2\varphi dx\bigg|\le\varepsilon|\nabla^2\varphi|_2^2+c/\varepsilon|v|_6^2|\rot\psi|_3^2.
$$
Using this in (\ref{3.2}) implies
\begin{equation}\eqal{
&{d\over dt}|\nabla\varphi|_2^2+\mu|\nabla^2\varphi|_2^2+\nu|\Delta\varphi|_2^2\le {c\over\nu}|\rot\psi|_3^2|v|_6^2\cr
&\le{c\over\nu}|\rot\psi|_3^2D_1^2.\cr}
\label{3.3}
\end{equation}
Inegrating with respect to time gives
\begin{equation}
|\nabla\varphi|_2^2+\mu|\nabla^2\varphi|_{2,\Omega^t}^2+ \nu|\Delta\varphi|_{2,\Omega^t}^2\le{c\over\nu}|\rot\psi|_{3,2,\Omega^t}^2D_1^2+ |\nabla\varphi(0)|_2^2.
\label{3.4}
\end{equation}
Multiply (\ref{1.7}) by $\rot\psi$ and integrate over $\Omega$. Then we obtain
\begin{equation}\eqal{
&{1\over2}{d\over dt}|\rot\psi|_2^2+\mu|\nabla\rot\psi|_2^2=-\intop_\Omega\rot\psi\cdot\nabla v\cdot\rot\psi dx\cr
&=\intop_\Omega\rot\psi\cdot\nabla\rot\psi\cdot vdx\cr}
\label{3.5}
\end{equation}
and the r.h.s. is bounded by
$$
\varepsilon|\nabla\rot\psi|_2^2+c/\varepsilon|\rot\psi|_3^2|v|_6^2.
$$
Using this in (\ref{3.5}) and integrating the result with respect to time we obtain
\begin{equation}
|\rot\psi(t)|_2^2+\mu|\nabla\rot\psi|_{2,\Omega^t}^2\le c|\rot\psi|_{3,2,\Omega^t}^2D_1^2+|\rot\psi(0)|_2^2.
\label{3.6}
\end{equation}
Multiplying (\ref{3.6}) by $1/\nu$ and adding to (\ref{3.4}) give
\begin{equation}\eqal{
&|\nabla\varphi(t)|_2^2+{1\over\nu}|\rot\psi(t)|_2^2+\mu \bigg(|\nabla^2\varphi|_{2,\Omega^t}^2+ {1\over\nu}|\nabla\rot\psi|_{2,\Omega^t}^2\bigg)\cr
&\quad+\nu|\Delta\varphi|_{2,\Omega^t}^2\le{c\over\nu}|\rot\psi|_{3,2,\Omega^t}^2 D_1^2+|\nabla\varphi(0)|_2^2+{1\over\nu}|\rot\psi(0)|_2^2.\cr}
\label{3.7}
\end{equation}
In view of interpolation the first term on the r.h.s. of (\ref{3.7}) is bounded by
$$
{c\over\nu}|\nabla\rot\psi|_{2,\Omega^t}|\rot\psi|_{2,\Omega^t}D_1^2\le {\varepsilon\over\nu}|\nabla\rot\psi|_{2,\Omega^t}^2
+{c\over\nu\varepsilon}|\rot\psi|_{2,\Omega^t}^2D_1^4
$$
where the last expression is bounded by
$$
{c\over\nu\varepsilon}|v|_{2,\Omega^t}^2D_1^4\le{c\over\nu\varepsilon}A_1^2D_1^4,
$$
where Lemma \ref{l2.1} is exploited.

\noindent
Using the estimates in (\ref{3.7}) we derive (\ref{3.1}). This concludes the proof.
\end{proof}

\begin{lemma}\label{l3.2}
Assume that $\nu\in(0,\infty)$, $D_1$, $D_2$, $\chi_0$, $\Psi$ are finite, $\nabla\varphi_t(0),\break \rot\psi_t(0)\in L_2(\Omega)$. Then
\begin{equation}\eqal{
&|\nabla\varphi_t(t)|_2^2+{1\over\nu}|\rot\psi_t(t)|_2^2+\mu\bigg( |\nabla^2\varphi_t|_{2,\Omega^t}^2+{1\over\nu}|\nabla\rot\psi_t|_{2,\Omega^t}^2\bigg)\cr
&\quad+\nu|\Delta\varphi_t|_{2,\Omega^t}^2\le{c\over\nu}\bigg[\bigg(D_1^2+ {\chi_0^2\over\nu}\bigg)\bigg(D_2^2+{\Psi^2\over\nu^2}\bigg)+D_2^2\bigg(D_1^2+ {\chi_0^2\over\nu}\bigg)\bigg]\cr
&\quad+|\nabla\varphi_t(0)|_2^2+{1\over\nu}|\rot\psi_t(0)|_2^2\equiv{c\over\nu}\phi_1 \bigg(D_1,D_2,{\chi_0\over\sqrt{\nu}},{\Psi\over\nu}\bigg)+|\nabla\varphi_t(0)|_2^2\cr
&\quad+{1\over\nu}|\rot\psi_t(0)|_2^2.\cr}
\label{3.8}
\end{equation}
\end{lemma}

\begin{proof}
Differentiate (\ref{1.7}) with respect to time, multiply by $\nabla\varphi_t$ and integrate over $\Omega$. Then we have
\begin{equation}\eqal{
&{1\over2}{d\over dt}|\nabla\varphi_t|_2^2+\mu|\nabla^2\varphi_t|_2^2+\nu |\Delta\varphi_t|_2^2\cr
&=-\intop_\Omega\rot\psi\cdot\nabla v_t\cdot\nabla\varphi_tdx-\intop_\Omega\rot\psi_t\cdot\nabla v\cdot\nabla\varphi_tdx\equiv I_1+I_2.\cr}
\label{3.9}
\end{equation}
Now we estimate the particular terms from the r.h.s. of (\ref{3.9}).
$$\eqal{
|I_1|&=\bigg|\intop_\Omega\rot\psi\cdot\nabla(\nabla\varphi_t+\rot\psi_t)\cdot \nabla\varphi_tdx\bigg|=\bigg|\intop_\Omega\rot\psi\cdot\nabla\rot\psi_t\cdot \nabla\varphi_tdx\bigg|\cr
&=\bigg|\intop_\Omega\rot\psi\cdot\nabla\nabla\varphi_t\cdot\rot\psi_tdx\bigg|\le \varepsilon|\nabla^2\varphi_t|_2^2+c/\varepsilon|\rot\psi_t|_3^2|\rot\psi|_6^2\cr
&\le\varepsilon|\nabla^2\varphi_t|_2^2+c/\varepsilon|\rot\psi_t|_3^2\bigg(D_1^2+ {\chi_0^2\over\nu}\bigg),\cr}
$$
where we used that $|\rot\psi|_6=|\rot\psi+\nabla\varphi-\nabla\varphi|_6\le|v|_6+|\nabla\varphi|_6\le D_1+{\chi_0\over\sqrt{\nu}}$, and
$$
|I_2|=\bigg|\intop_\Omega\rot\psi_t\cdot\nabla\nabla\varphi_t\cdot vdx\bigg|\le \varepsilon|\nabla^2\varphi_t|_2^2+c/\varepsilon|\rot\psi_t|_3^2D_1^2.
$$
Using the estimates in (\ref{3.9}) and integrating the result with respect to time yield
\begin{equation}\eqal{
&|\nabla\varphi_t(t)|_2^2+\mu|\nabla^2\varphi_t|_{2,\Omega^t}^2+ \nu|\Delta\varphi_t|_{2,\Omega^t}^2\cr
&\le{c\over\nu}\bigg(D_1^2+{\chi_0^2\over\nu}\bigg)|\rot\psi_t|_{3,2,\Omega^t}^2+ |\nabla\varphi_t(0)|_2^2.\cr}
\label{3.10}
\end{equation}
Differentiate (\ref{1.7}) with respect to time, multiply by $\rot\psi_t$ and integrate over $\Omega$. Then we obtain
\begin{equation}\eqal{
&{1\over2}{d\over dt}|\rot\psi_t|_2^2+\mu|\nabla\rot\psi_t|_2^2=-\intop_\Omega \rot\psi\cdot\nabla v_t\cdot\rot\psi_tdx\cr
&\quad-\intop_\Omega\rot\psi_t\cdot\nabla v\cdot\rot\psi_tdx\equiv J_1+J_2,\cr}
\label{3.11}
\end{equation}
where
$$
|J_1|=\bigg|\intop_\Omega\rot\psi\cdot\nabla\rot\psi_t\cdot v_tdx\bigg|\le \varepsilon|\nabla\rot\psi_t|_2^2+c/\varepsilon|\rot\psi|_3^2|v_t|_6^2
$$
and
$$
|J_2|=\bigg|\intop_\Omega\rot\psi_t\cdot\nabla\rot\psi_t\cdot vdx\bigg|\le \varepsilon|\nabla\rot\psi_t|_2^2+c/\varepsilon|\rot\psi_t|_3^2|v|_6^2.
$$
Using the estimates in (\ref{3.11}), integrating the result with respect to time and employing (\ref{2.27}), Lemma \ref{l2.6} we obtain
\begin{equation}\eqal{
&|\rot\psi_t|_2^2+\mu|\nabla\rot\psi_t|_{2,\Omega^t}^2\le c|\rot\psi|_{3,\infty,\Omega^t}^2D_2^2\cr
&\quad+c|\rot\psi_t|_{3,2,\Omega^t}^2D_1^2+|\rot\psi_t(0)|_2^2.\cr}
\label{3.12}
\end{equation}
Multiplying (\ref{3.12}) by $1/\nu$ and adding to (\ref{3.10}) yield
\begin{equation}\eqal{
&|\nabla\varphi_t(t)|_2^2+{1\over\nu}|\rot\psi_t(t)|_2^2+\mu\bigg( |\nabla^2\varphi_t|_{2,\Omega^t}^2+{1\over\nu}|\nabla\rot\psi_t|_{2,\Omega^t}^2\bigg)\cr
&\quad+\nu|\Delta\varphi_t|_{2,\Omega^t}^2\le{c\over\nu}\bigg(D_1^2+ {\chi_0^2\over\nu}\bigg)|\rot\psi_t|_{3,2,\Omega^t}^2\cr
&\quad+{c\over\nu}|\rot\psi|_{3,\infty,\Omega^t}^2D_2^2+ |\nabla\varphi_t(0)|_2^2+{1\over\nu}|\rot\psi_t(0)|_2^2.\cr}
\label{3.13}
\end{equation}
Employing the inequalities
$$\eqal{
|\rot\psi_t|_{3,2,\Omega^t}^2&=|\rot\psi_t+\nabla\varphi_t- \nabla\varphi_t|_{3,2,\Omega^t}^2\le c(|v_t|_{3,2,\Omega^t}^2+|\nabla\varphi_t|_{3,2,\Omega^t}^2)\cr
&\le c\bigg(D_2^2+{\Psi^2\over\nu^2}\bigg),\cr}
$$
$$\eqal{
|\rot\psi|_{3,\infty,\Omega^t}^2&=|\rot\psi+\nabla\varphi- \nabla\varphi|_{3,\infty,\Omega^t}^2\le c(|v|_{3,\infty,\Omega^t}^2+|\nabla\varphi|_{3,\infty,\Omega^t}^2)\cr
&\le c\bigg(D_1^2+{\chi_0^2\over\nu}\bigg)\cr}
$$
in (\ref{3.13}) implies (\ref{3.8}) and concludes the proof.
\end{proof}

\begin{lemma}\label{l3.3}
Assume that $\nu\in(0,\infty)$, $A_1$, $D_1$ are finite, $\nabla\varphi(0)$, $\rot\psi(0)\in H^1(\Omega)$. Then
\begin{equation}\eqal{
&|\nabla\varphi_x(t)|_2^2+{1\over\nu}|\rot\psi_x(t)|_2^2+\mu\bigg( |\nabla^2\varphi_x|_{2,\Omega^t}^2+{1\over\nu}|\nabla\rot\psi_x|_{2,\Omega^t}^2\bigg)\cr
&\quad+\nu|\Delta\varphi_x|_{2,\Omega^t}^2\le{c\over\nu}A_1^2\bigg(D_1^2+ {\chi_0^2\over\nu}\bigg)^2+{c\over\nu}D_1^4\bigg[A_1^2+{1\over\nu}A_1^2D_1^4\cr
&\quad+|\nabla\varphi(0)|_2^2+{1\over\nu}|\rot\psi(0)|_2^2\bigg]+|\nabla\varphi_x(0)|_2^2+ {1\over\nu}|\rot\psi_x(0)|_2^2\cr
&\equiv{c\over\nu}\phi_2\bigg(A_1,D_1,|\nabla\varphi(0)|_2^2+ {1\over\nu}|\rot\psi(0)|_2^2\bigg)+ |\nabla\varphi_x(0)|_2^2+{1\over\nu}|\rot\psi_x(0)|_2^2.\cr}
\label{3.14}
\end{equation}
\end{lemma}

\begin{proof}
Differentiate (\ref{1.7}) with respect to $x$, multiply by $\nabla\varphi_x$ and integrate over $\Omega$. Then we obtain
\begin{equation}\eqal{
&{1\over2}{d\over dt}|\nabla\varphi_x|_2^2+\mu|\nabla^2\varphi_x|_2^2+\nu |\Delta\varphi_x|_2^2\cr
&=-\intop_\Omega\rot\psi\cdot\nabla v_x\cdot\nabla\varphi_xdx-\intop_\Omega\rot\psi_x\cdot\nabla v\cdot\nabla\varphi_xdx\equiv I_1+I_2.\cr}
\label{3.15}
\end{equation}
Integration by parts in $I_1$ yields
$$
|I_1|=\bigg|\intop_\Omega\rot\psi\cdot\nabla\nabla\varphi_x\cdot v_xdx\bigg|\le\varepsilon|\nabla^2\varphi_x|_2^2+c/\varepsilon|\rot\psi|_6^2|v_x|_3^2.
$$
Similarly, we have
$$
|I_2|=\bigg|\intop_\Omega\rot\psi_x\cdot\nabla\nabla\varphi_x\cdot vdx\bigg|\le\varepsilon|\nabla^2\varphi_x|_2^2+c/\varepsilon|\rot\psi_x|_3^2|v|_6^2.
$$
Using the estimates in (\ref{3.15}) and integrating the result with respect to time imply
\begin{equation}\eqal{
&|\nabla\varphi_x(t)|_2^2+\mu|\nabla^2\varphi_x|_{2,\Omega^t}^2+ \nu|\Delta\varphi_x|_{2,\Omega^t}^2\le{c\over\nu}|\rot\psi|_{6,\infty,\Omega^t}^2 |v_x|_{3,2,\Omega^t}^2\cr
&\quad+{c\over\nu}|\rot\psi_x|_{3,2,\Omega^t}^2D_1^2+|\nabla\varphi_x(0)|_2^2.\cr}
\label{3.16}
\end{equation}
To estimate the first term on the r.h.s. of (\ref{3.16}) we use the estimate
\begin{equation}\eqal{
|\rot\psi|_{6,\infty,\Omega^t}&=|\rot\psi+\nabla\varphi- \nabla\varphi|_{6,\infty,\Omega^t}\cr
&\le c(|v|_{6,\infty,\Omega^t}+|\nabla\varphi|_{6,\infty,\Omega^t})\le c\bigg(D_1+{\chi_0\over\sqrt{\nu}}\bigg).}
\label{3.17}
\end{equation}
Then (\ref{3.16}) takes the form
\begin{equation}\eqal{
&|\nabla\varphi_x(t)|_2^2+\mu|\nabla^2\varphi_x|_{2,\Omega^t}^2+\nu |\Delta\varphi_x|_{2,\Omega^t}^2\le{c\over\nu}|v_x|_{3,2,\Omega^t}^2 \bigg(D_1^2+{\chi_0^2\over\nu}\bigg)\cr
&\quad+{c\over\nu}|\rot\psi_x|_{3,2,\Omega^t}^2D_1^2+|\nabla\varphi_x(0)|_2^2.\cr}
\label{3.18}
\end{equation}
Differentiate (\ref{1.7}) with respect to $x$, multiply by $\rot\psi_x$ and integrate over $\Omega$. Then we obtain
\begin{equation}\eqal{
&{1\over2}{d\over dt}|\rot\psi_x|_2^2+\mu|\nabla\rot\psi_x|_2^2=-\intop_\Omega \rot\psi\cdot\nabla v_x\cdot\rot\psi_xdx\cr
&\quad-\intop_\Omega\rot\psi_x\cdot\nabla v\cdot\rot\psi_xdx\equiv J_1+J_2.\cr}
\label{3.19}
\end{equation}
Integrating by parts in $J_1$ yields
$$
|J_1|=\bigg|\intop_\Omega\rot\psi\cdot\nabla\rot\psi_x\cdot v_xdx\bigg|\le\varepsilon|\nabla\rot\psi_x|_2^2+{c\over\varepsilon} |\rot\psi|_6^2|v_x|_3^2.
$$
Similarly, we have
$$
|J_2|=\bigg|\intop_\Omega\rot\psi_x\cdot\nabla\rot\psi_x\cdot vdx\bigg|\le \varepsilon|\nabla\rot\psi_x|_2^2+{c\over\varepsilon}|\rot\psi_x|_3^2|v|_6^2.
$$
Using the estimates and (\ref{3.17}) in (\ref{3.19}), taking into account Lemma \ref{l2.4} and integrating the result with respect to time we obtain the inequality
\begin{equation}\eqal{
&|\rot\psi_x(t)|_2^2+\mu|\nabla\rot\psi_x|_{2,\Omega^t}^2\le c|v_x|_{3,2,\Omega^t}^2\bigg(D_1^2+{\chi_0^2\over\nu}\bigg)\cr
&\quad+c|\rot\psi_x|_{3,2,\Omega^t}^2D_1^2+|\rot\psi_x(0)|_2^2.\cr}
\label{3.20}
\end{equation}
Multiplying (\ref{3.20}) by $1/\nu$ and summing up with (\ref{3.18}) we have
\begin{equation}\eqal{
&|\nabla\varphi_x(t)|_2^2+{1\over\nu}|\rot\psi_x(t)|_2^2+\mu\bigg( |\nabla^2\varphi_x|_{2,\Omega^t}^2+{1\over\nu}|\nabla\rot\psi_x|_{2,\Omega^t}^2\bigg)\cr
&\quad+\nu|\Delta\varphi_x|_{2,\Omega^t}^2\le{c\over\nu}|v_x|_{3,2,\Omega^t}^2 \bigg(D_1^2+{\chi_0^2\over\nu}\bigg)\cr
&\quad+{c\over\nu}|\rot\psi_x|_{3,2,\Omega^t}^2D_1^2+|\nabla\varphi_x(0)|_2^2+ {1\over\nu}|\rot\psi_x(0)|_2^2.\cr}
\label{3.21}
\end{equation}
By interpolation and Lemma \ref{l2.1} we have
\begin{equation}
|v_x|_{3,2,\Omega^t}^2\le c|\nabla v_x|_{2,\Omega^t}|v_x|_{2,\Omega^t}\le \varepsilon|\nabla v_x|_{2,\Omega^t}^2+c/\varepsilon A_1^2
\label{3.22}
\end{equation}
and
\begin{equation}\eqal{
|\rot\psi_x|_{3,2,\Omega^t}^2&\le c|\nabla\rot\psi_x|_{2,\Omega^t}|\rot\psi_x|_{2,\Omega^t}\cr
&\le\varepsilon|\nabla\rot\psi_x|_{2,\Omega^t}^2+ c/\varepsilon|\rot\psi_x|_{2,\Omega^t}^2,\cr}
\label{3.23}
\end{equation}
where
$$\eqal{
|\rot\psi_x|_{2,\Omega^t}^2&=|\rot\psi_x+\nabla\varphi_x-\nabla\varphi_x|_{2,\Omega^t}\le c(|v_x|_{2,\Omega^t}^2+|\nabla\varphi_x|_{2,\Omega^t}^2)\cr
&\le cA_1^2+c|\nabla\varphi_x|_{2,\Omega^t}^2\cr}
$$
and $|\nabla\varphi_x|_{2,\Omega^t}$ is estimated by (\ref{3.1}). Hence, we have
$$
|\nabla\varphi_x|_{2,\Omega^t}^2\le{c\over\nu}A_1^2D_1^4+|\nabla\varphi(0)|_2^2+ {1\over\nu}|\rot\psi(0)|_2^2.
$$
Using the above estimates in (\ref{3.21}) implies (\ref{3.14}). This concludes the proof.
\end{proof}

\begin{lemma}\label{l3.4}
Assume that $\nu\in(0,\infty)$, $D_1$, $D_2$, $\chi_0$ are finite, $\nabla\varphi(0),\break\rot\psi(0)\in H^1(\Omega)$, $\nabla\varphi_t(0),\rot\psi_t(0)\in H^1(\Omega)$, $\rot\psi_x\in L_\infty(0,t:L_3(\Omega))$.\\
Then
\begin{equation}\eqal{
&|\nabla\varphi_{xt}(t)|_2^2+{1\over\nu}|\rot\psi_{xt}(t)|_2^2+\mu \bigg(|\nabla^2\varphi_{xt}|_{2,\Omega^t}^2+{1\over\nu} |\nabla\rot\psi_{xt}|_{2,\Omega^t}^2\bigg)\cr
&\quad+\nu|\Delta\varphi_{xt}|_{2,\Omega^t}^2\le{c\over\nu}D_2^2 |\rot\psi_x|_{3,\infty,\Omega^t}^2+{c\over\nu}\bigg(D_1^4+ {\chi_0^4\over\nu^2}\bigg)\cdot\cr
&\quad\cdot(\phi_1+\nu|\nabla\varphi_t(0)|_2^2+|\rot\psi_t(0)|_2^2)+{c\over\nu} \bigg(D_2^2+{\chi_0^2\over\nu}\bigg)(\phi_2+\nu|\nabla\varphi_x(0)|_2^2\cr
&\quad+|\rot\psi_x(0)|_2^2)+{c\over\nu}\bigg(D_1^2+{\chi_0^2\over\nu}\bigg) {\Psi^2\over\nu^2}+|\nabla\varphi_{xt}(0)|_2^2+{1\over\nu} |\rot\psi_{xt}(0)|_2^2.\cr}
\label{3.24}
\end{equation}
\end{lemma}

\begin{proof}
Differentiate (\ref{1.7}) with respect to $x$ and $t$, multiply the result by $\nabla\varphi_{xt}$ and integrate over $\Omega$. Then we have
\begin{equation}\eqal{
&{1\over2}{d\over dt}|\nabla\varphi_{xt}|_2^2+\mu|\nabla^2\varphi_{xt}|_2^2+\nu |\Delta\varphi_{xt}|_2^2\cr
&=-\intop_\Omega(\rot\psi\cdot\nabla v)_{xt}\cdot\nabla\varphi_{xt}dx\equiv -I.\cr}
\label{3.25}
\end{equation}
Performing differentiations with respect to $x$ and $t$ in $I$ implies
$$
I=\intop_\Omega(\rot\psi_{xt}\cdot\nabla v+\rot\psi_x\cdot\nabla v_t+\rot\psi_t\cdot\nabla v_x+\rot\psi\cdot\nabla v_{xt})\cdot\nabla\varphi_{xt}dx\equiv\sum_{i=1}^4I_i.
$$
Integrating by parts in $I_1$ yields
$$
|I_1|=\bigg|\intop_\Omega\rot\psi_{xt}\cdot\nabla\nabla\varphi_{xt}\cdot vdx\bigg|\le\varepsilon|\nabla^2\varphi_{xt}|_2^2+ c/\varepsilon|\rot\psi_{xt}|_3^2|v|_6^2.
$$
Similarly, integration by parts in $I_2$ implies
$$
|I_2|=\bigg|\intop_\Omega\rot\psi_x\cdot\nabla\nabla\varphi_{xt}\cdot v_tdx\bigg|\le\varepsilon|\nabla^2\varphi_{xt}|_2^2+c/\varepsilon|\rot\psi_x|_3^2 |v_t|_6^2.
$$
Next, we have
$$
|I_3|=\bigg|\intop_\Omega\rot\psi_t\cdot\nabla\nabla\varphi_{xt}\cdot v_xdx\bigg|\le\varepsilon|\nabla^2\varphi_{xt}|_2^2+c/\varepsilon|\rot\psi_t|_3^2 |v_x|_6^2.
$$
Finally, we consider
$$\eqal{
I_4&=\intop_\Omega\rot\psi\cdot\nabla(\nabla\varphi_{xt}+\rot\psi_{xt})\cdot \nabla\varphi_{xt}dx\cr
&=\intop_\Omega\rot\psi\cdot\nabla\rot\psi_{xt}\cdot\nabla\varphi_{xt}dx= -\intop_\Omega(\rot\psi)_i(\rot\psi_{xt})_j\partial_{x_i}\partial_{x_j} \varphi_{xt}dx.\cr}
$$
Then, we have
$$
|I_4|\le\varepsilon|\nabla^2\varphi_{xt}|_2^2+c/\varepsilon |\rot\psi|_6^2|\rot\psi_{xt}|_3^2.
$$
Using the above estimates in (\ref{3.25}) we obtain after integration with respect to time and with the help of (\ref{2.27}), (\ref{2.28}) and (\ref{3.17}) the inequality
\begin{equation}\eqal{
&|\nabla\varphi_{xt}(t)|_2^2+\mu|\nabla^2\varphi_{xt}|_{2,\Omega^t}^2+\nu |\Delta\varphi_{xt}|_{2,\Omega^t}^2\le{c\over\nu}D_1^2|\rot\psi_{xt}|_{3,2,\Omega^t}^2\cr
&\quad+{c\over\nu}D_2^2|\rot\psi_x|_{3,\infty,\Omega^t}^2+{c\over\nu} |\rot\psi_t|_{3,\infty,\Omega^t}^2|v_x|_{6,2,\Omega^t}^2\cr
&\quad+{c\over\nu}\bigg(D_1^2+{\chi_0^2\over\nu}\bigg)|\rot\psi_{xt}|_{3,2,\Omega^t}^2 +|\nabla\varphi_{xt}(0)|_2^2.\cr}
\label{3.26}
\end{equation}
Using (\ref{3.14}) we have
\begin{equation}\eqal{
|v_x|_{6,2,\Omega^t}^2&\le c\|v_x\|_{1,2,\Omega^t}^2\le c(\nu \|\nabla\varphi_x\|_{1,2,\Omega^t}^2+\|\rot\psi_x\|_{1,2,\Omega^t}^2)\cr
&\le c\phi_2+c(\nu|\nabla\varphi_x(0)|_2^2+|\rot\psi_x(0)|_2^2).\cr}
\label{3.27}
\end{equation}
Therefore, (\ref{3.26}) takes the form
\begin{equation}\eqal{
&|\nabla\varphi_{xt}(t)|_2^2+\mu|\nabla^2\varphi_{xt}|_{2,\Omega^t}^2+\nu |\Delta\varphi_{xt}|_{2,\Omega^t}^2\cr
&\le{c\over\nu}\bigg(D_1^2+{\chi_0^2\over\nu}\bigg) |\rot\psi_{xt}|_{3,2,\Omega^t}^2+{c\over\nu}D_2^2|\rot\psi_x|_{3,\infty,\Omega^t}^2\cr
&\quad+{c\over\nu} |\rot\psi_t|_{3,\infty,\Omega^t}^2(\phi_2+\nu|\nabla\varphi_x(0)|_2^2+|\rot\psi_x(0)|_2^2)+|\nabla\varphi_{xt}(0)|_2^2.\cr}
\label{3.28}
\end{equation}
Differentiate (\ref{1.7}) with respect to $x$ and $t$, multiply the result by $\rot\psi_{xt}$ and integrate over $\Omega$. Then we have
\begin{equation}
{1\over2}{d\over dt}|\rot\psi_{xt}(t)|_2^2+\mu|\nabla\rot\psi_{xt}|_2^2= -\intop_\Omega(\rot\psi\cdot\nabla v)_{xt}\cdot\rot\psi_{xt}dx\equiv -J.
\label{3.29}
\end{equation}
Performing differentiations with respect to $x$ and $t$ in $J$ yields
$$
J=\intop_\Omega(\rot\psi_{xt}\cdot\nabla v+\rot\psi_x\cdot\nabla v_t+\rot\psi_t\cdot\nabla v_x+\rot\psi\cdot\nabla v_{xt})\cdot\rot\psi_{xt}dx\equiv \sum_{i=1}^4J_i.
$$
Integrating by parts in $J_1$ yields
$$\eqal{
|J_1|&=\bigg|\intop_\Omega\rot\psi_{xt}\cdot\nabla\rot\psi_{xt}\cdot vdx\bigg|\le\varepsilon|\nabla\rot\psi_{xt}|_2^2\cr
&\quad+c/\varepsilon|\rot\psi_{xt}|_3^2|v|_6^2.\cr}
$$
Similarly, integration by parts in $J_2$ implies
$$
|J_2|=\bigg|\intop_\Omega\rot\psi_x\cdot\nabla\rot\psi_{xt}\cdot v_tdx\bigg|\le\varepsilon|\nabla\rot\psi_{xt}|_2^2+{c\over\varepsilon} |\rot\psi_x|_3^2|v_t|_6^2.
$$
Next, we have
$$
|J_3|=\bigg|\intop_\Omega\rot\psi_t\cdot\nabla\rot\psi_{xt}\cdot v_xdx\bigg|\le\varepsilon|\nabla\rot\psi_{xt}|_2^2+{c\over\varepsilon} |\rot\psi_t|_3^2|v_x|_6^2.
$$
Finally, we examine
$$\eqal{
J_4&=\intop_\Omega\rot\psi\cdot\nabla(\nabla\varphi_{xt}+\rot\psi_{xt})\cdot \rot\psi_{xt}dx\cr
&=\intop_\Omega\rot\psi\cdot\nabla\nabla\varphi_{xt}\cdot\rot\psi_{xt}dx= -\intop_\Omega\rot\psi\cdot\nabla\rot\psi_{xt}\cdot\nabla\varphi_{xt}dx.\cr}
$$
Then we have
$$
|J_4|\le\varepsilon|\nabla\rot\psi_{xt}|_2^2+c/\varepsilon|\rot\psi|_6^2 |\nabla\varphi_{xt}|_3^2.
$$
Using the above estimates in (\ref{3.29}) we obtain after integration with respect to time with the help of (\ref{2.27}), (\ref{2.28}), (\ref{3.17}) and (\ref{3.27}) the inequality
\begin{equation}\eqal{
&|\rot\psi_{xt}(t)|_2^2+\mu|\nabla\rot\psi_{xt}|_{2,\Omega^t}^2\le cD_1^2|\rot\psi_{xt}|_{3,2,\Omega^t}^2\cr
&\quad+cD_2^2|\rot\psi_x|_{3,\infty,\Omega^t}^2+c|\rot\psi_t|_{3,\infty,\Omega^t}^2 (\phi_2+\nu|\nabla\varphi_x(0)|_2^2\cr
&\quad+|\rot\psi_x(0)|_2^2)+c\bigg(D_1^2+{\chi_0^2\over\nu}\bigg) |\nabla\varphi_{xt}|_{3,2,\Omega^t}^2+|\rot\psi_{xt}(0)|_2^2.\cr}
\label{3.30}
\end{equation}
Multiplying (\ref{3.30}) by $1/\nu$ and summing up with (\ref{3.28}) we obtain
\begin{equation}\eqal{
&|\nabla\varphi_{xt}(t)|_2^2+{1\over\nu}|\rot\psi_{xt}(t)|_2^2+\mu \bigg(|\nabla^2\varphi_{xt}|_{2,\Omega^t}^2+{1\over\nu} |\nabla\rot\psi_{xt}|_{2,\Omega^t}^2\bigg)\cr
&\quad+\nu|\Delta\varphi_{xt}|_{2,\Omega^t}^2\le {c\over\nu}D_2^2|\rot\psi_x|_{3,\infty,\Omega^t}^2+{c\over\nu} |\rot\psi_t|_{3,\infty,\Omega^t}^2(\phi_2+\nu|\nabla\varphi_x(0)|_2^2\cr
&\quad+|\rot\psi_x(0)|_2^2)+{c\over\nu}\bigg(D_1^2+{\chi_0^2\over\nu}\bigg)(|\nabla\varphi_{xt}|_{3,2,\Omega^t}^2 +|\rot\psi_{xt}|_{3,2,\Omega^t}^2)\cr
&\quad+ |\nabla\varphi_{xt}(0)|_2^2+{1\over\nu}|\rot\psi_{xt}(0)|_2^2.\cr}
\label{3.31}
\end{equation}
Now we shall estimate the unknown quantities appeared on the r.h.s. of (\ref{3.31}). We need the intepolation
$$
|\rot\psi_{xt}|_{3,2,\Omega^t}^2\le\varepsilon|\nabla\rot\psi_{xt}|_{2,\Omega^t}^2+ c/\varepsilon|\rot\psi_{xt}|_{2,\Omega^t}^2,
$$
where, in view of (\ref{3.8}), we have
$$
|\rot\psi_{xt}|_{2,\Omega^t}^2\le c\phi_1+\nu|\nabla\varphi_t(0)|_2^2+|\rot\psi_t(0)|_2^2.
$$
Therefore, the following part of the third term on the r.h.s. of (\ref{3.31}) is bounded by
$$\eqal{
&{c\over\nu}\bigg(D_1^2+{\chi_0^2\over\nu}\bigg)|\rot\psi_{xt}|_{3,2,\Omega^t}^2\cr
&\le{\varepsilon\over\nu}|\nabla\rot\psi_{xt}|_{2,\Omega^t}^2+ {c\over\nu\varepsilon}(\phi_1+\nu|\nabla\varphi_t(0)|_2^2+ |\rot\psi_t(0)|_2^2)\bigg(D_1^4+{\chi_0^4\over\nu^2}\bigg).}
$$
To estimate the second term on the r.h.s. of (\ref{3.31}) we use the interpolation
$$
|\rot\psi_t|_{3,\infty,\Omega^t}^2\le\varepsilon |\nabla\rot\psi_t|_{2,\infty,\Omega^t}^2+c/\varepsilon |\rot\psi_t|_{2,\infty,\Omega^t}^2,
$$
where
$$\eqal{
|\rot\psi_t|_{2,\infty,\Omega^t}^2&=|\rot\psi_t+\nabla\varphi_t- \nabla\varphi_t|_{2,\infty,\Omega^t}^2\le c(|v_t|_{2,\infty,\Omega^t}^2+ |\nabla\varphi_t|_{2,\infty,\Omega^t}^2)\cr
&\le c\bigg(D_2^2+{\chi_0^2\over\nu}\bigg).\cr}
$$
Then the second term on the r.h.s. of (\ref{3.31}) is bounded by
$$\eqal{
&{\varepsilon\over\nu}|\nabla\rot\psi_t|_{2,\infty,\Omega^t}^2+{c\over\nu\varepsilon} \bigg(D_2^2+{\chi_0^2\over\nu}\bigg)(\phi_2+\nu|\nabla\varphi_x(0)|_2^2\cr
&\quad+|\rot\psi_x(0)|_2^2)^2.\cr}
$$
The remainintg part of the third term on the r.h.s. of (\ref{3.31}) is bounded by
$$
{c\over\nu}\bigg(D_1^2+{\chi_0^2\over\nu}\bigg){\Psi^2\over\nu^2}.
$$
Using the above estimates in (\ref{3.31}) and assuming that $\varepsilon$ is sufficiently small we derive (\ref{3.24}). This concludes the proof.
\end{proof}

\begin{lemma}\label{l3.5}
Assume that $\nu\in(0,\infty)$, $A_1$, $D_1$, $D_2$, $\chi_0$, $\Psi$ are finite. Let $\nabla\varphi(0),\rot\psi(0)\in H^2(\Omega)$, $\nabla\varphi_t(0),\rot\psi_t(0)\in H^1(\Omega)$. Then
\begin{equation}\eqal{
&|\nabla\varphi_{xx}(t)|_2^2+{1\over\nu}|\rot\psi_{xx}(t)|_2^2+\mu \bigg( |\nabla^2\varphi_{xx}|_{2,\Omega^t}^2+{1\over\nu} |\nabla\rot\psi_{xx}|_{2,\Omega^t}^2\bigg)\cr
&\quad+\nu|\Delta\varphi_{xx}|_{2,\Omega^t}^2\le{c\over\nu}\phi_4+ |\nabla\varphi_{xx}(0)|_2^2+{1\over\nu}|\rot\psi_{xx}(0)|_2^2,\cr}
\label{3.32}
\end{equation}
where $\phi_4$ is defined in (\ref{3.39}).
\end{lemma}

\begin{proof}
Differentiate (\ref{1.7}) twice with respect to $x$, multiply by $\nabla\varphi_{xx}$ and integrate over $\Omega$. Then we obtain
\begin{equation}
{1\over2}{d\over dt}|\nabla\varphi_{xx}|_2^2+\mu|\nabla^2\varphi_{xx}|_2^2+\nu |\Delta\varphi_{xx}|_2^2=-\intop_\Omega(\rot\psi\cdot\nabla v)_{xx}\cdot\nabla\varphi_{xx}dx\equiv -I.
\label{3.33}
\end{equation}
Carrying out differentiations in $I$ yields
$$\eqal{
I&=\intop_\Omega(\rot\psi_{xx}\cdot\nabla v+2\rot\psi_x\cdot\nabla v_x+\rot\psi\cdot\nabla v_{xx})\cdot\nabla\varphi_{xx}dx\cr
&\equiv I_1+I_2+I_3.\cr}
$$
Integrating by parts in $I_1$ implies
$$
|I_1|=\bigg|\intop_\Omega\rot\psi_{xx}\cdot\nabla\nabla\varphi_{xx}\cdot vdx\bigg|\le\varepsilon|\nabla^2\varphi_{xx}|_2^2+{c\over\varepsilon} |\rot\psi_{xx}|_3^2|v|_6^2,
$$
where, in view of (\ref{2.27}), the second term is bounded by
$$
{c\over\varepsilon}|\rot\psi_{xx}|_3^2D_1^2.
$$
We integrate by parts in $I_2$. Then we obtain
$$
|I_2|=2\bigg|\intop_\Omega\rot\psi_x\cdot\nabla^2\varphi_{xx}\cdot v_xdx\bigg|\le\varepsilon|\nabla^2\varphi_{xx}|_2^2+c/\varepsilon |\rot\psi_x|_3^2|v_x|_6^2.
$$
Finally,
$$\eqal{
I_3&=\intop_\Omega\rot\psi\cdot\nabla(\nabla\varphi_{xx}+\rot\psi_{xx})\cdot \nabla\varphi_{xx}dx\cr
&=\intop_\Omega\rot\psi\cdot\nabla\rot\psi_{xx}\cdot\nabla\varphi_{xx}dx.\cr}
$$
Integration by parts in $I_3$ yields
$$
|I_3|=\bigg|\intop_\Omega\rot\psi\cdot\nabla\nabla\varphi_{xx}\cdot \rot\psi_{xx}dx\bigg|\le\varepsilon|\nabla^2\varphi_{xx}|_2^2+c/\varepsilon |\rot\psi|_6^2|\rot\psi_{xx}|_3^2,
$$
where
\begin{equation}
|\rot\psi|_6=|\rot\psi+\nabla\varphi-\nabla\varphi|_6\le|v|_6+|\nabla\varphi|_6\le D_1+{\chi_0\over\sqrt{\nu}}.
\label{3.34}
\end{equation}
Using the estimates in (\ref{3.33}) and integrating the result with respect to time yield
$$\eqal{
&|\nabla\varphi_{xx}(t)|_2^2+\mu|\nabla^2\varphi_{xx}|_{2,\Omega^t}^2+\nu |\Delta\varphi_{xx}|_{2,\Omega^t}^2\le{c\over\nu}|\rot\psi_{xx}|_{3,2,\Omega^t}^2 D_1^2\cr
&\quad+{c\over\nu}|\rot\psi_x|_{3,\infty,\Omega^t}^2|v_x|_{6,2,\Omega^t}^2+ {c\over\nu}|\rot\psi_{xx}|_{3,2,\Omega^t}^2\bigg(D_1^2+{\chi_0^2\over\nu}\bigg)\cr
&\quad+|\nabla\varphi_{xx}(0)|_2^2.\cr}
$$
Finally, using (\ref{3.27}), we get
\begin{equation}\eqal{
&|\nabla\varphi_{xx}(t)|_2^2+\mu|\nabla^2\varphi_{xx}|_{2,\Omega^t}^2+\nu |\Delta\varphi_{xx}|_{2,\Omega^t}^2\cr
&\le{c\over\nu}|\rot\psi_{xx}|_{3,2,\Omega^t}^2\bigg(D_1^2+{\chi_0^2\over\nu}\bigg)+ {c\over\nu}|\rot\psi_x|_{3,\infty,\Omega^t}^2(\phi_2+\nu|\nabla\varphi_x(0)|_2^2\cr
&\quad+|\rot\psi_x(0)|_2^2)+|\nabla\varphi_{xx}(0)|_2^2.\cr}
\label{3.35}
\end{equation}
Differentiate (\ref{1.7}) twice with respect to $x$, multiply by $\rot\psi_{xx}$ and integrate over $\Omega$. Then we have
\begin{equation}
{1\over2}{d\over dt}|\rot\psi_{xx}|_2^2+\mu|\nabla\rot\psi_{xx}|_2^2= -\intop_\Omega(\rot\psi\cdot\nabla v)_{xx}\cdot\rot\psi_{xx}dx\equiv-J.
\label{3.36}
\end{equation}
Performing differentiations in $J$ implies
$$\eqal{
J&=\intop_\Omega(\rot\psi_{xx}\cdot\nabla v+2\rot\psi_x\cdot\nabla v_x+\rot\psi\cdot\nabla v_{xx})\cdot\rot\psi_{xx}dx\cr
&\equiv J_1+J_2+J_3.\cr}
$$
Integration by parts in $J_1$ gives
$$
|J_1|=\bigg|\intop_\Omega\rot\psi_{xx}\cdot\nabla\rot\psi_{xx}\cdot vdx\bigg|\le\varepsilon|\nabla\rot\psi_{xx}|_2^2+c/\varepsilon |\rot\psi_{xx}|_3^2|v|_6^2.
$$
Proceeding, we have
$$\eqal{
|J_2|&=2\bigg|\intop_\Omega\rot\psi_x\cdot\nabla\rot\psi_{xx}\cdot v_xdx\bigg|\le\varepsilon|\nabla\rot\psi_{xx}|_2^2\cr
&\quad+c/\varepsilon|\rot\psi_x|_3^2|v_x|_6^2.\cr}
$$
Finally,
$$\eqal{
J_3&=\intop_\Omega\rot\psi\cdot\nabla(\rot\psi_{xx}+\nabla\varphi_{xx})\cdot \rot\psi_{xx}dx\cr
&=\intop_\Omega\rot\psi\cdot\nabla\nabla\varphi_{xx}\cdot\rot\psi_{xx}dx= -\intop_\Omega\rot\psi\cdot\nabla\rot\psi_{xx}\cdot\nabla\varphi_{xx}dx.\cr}
$$
Hence, we have
$$
|J_3|\le\varepsilon|\nabla\rot\psi_{xx}|_2^2+{c\over\varepsilon} |\rot\psi|_6^2|\nabla\varphi_{xx}|_3^2.
$$
Using the above estimates in (\ref{3.36}) integrating the result with respect to time, exploiting (\ref{3.27}), (\ref{2.27}) and (\ref{3.34}) we obtain the inequality
\begin{equation}\eqal{
&|\rot\psi_{xx}(t)|_2^2+\mu|\nabla\rot\psi_{xx}|_{2,\Omega^t}^2\le c|\rot\psi_{xx}|_{3,2,\Omega^t}^2D_1^2\cr
&\quad+c|\rot\psi_x|_{3,\infty,\Omega^t}^2(\phi_2+\nu|\nabla\varphi_x(0)|_2^2+ |\rot\psi_x(0)|_2^2)\cr
&\quad+c|\nabla\varphi_{xx}|_3^2\bigg(D_1^2+{\chi_0^2\over\nu}\bigg)+ |\rot\psi_{xx}(0)|_2^2.\cr}
\label{3.37}
\end{equation}
Multiplying (\ref{3.37}) by $1/\nu$ and adding to (\ref{3.35}) imply
\begin{equation}\eqal{
&|\nabla\varphi_{xx}(t)|_2^2+{1\over\nu}|\rot\psi_{xx}(t)|_2^2+\mu \bigg(|\nabla^2\varphi_{xx}|_{2,\Omega^t}^2+{1\over\nu} |\nabla\rot\psi_{xx}|_{2,\Omega^t}^2\bigg)\cr
&\quad+\nu|\Delta\varphi_{xx}|_{2,\Omega^t}^2\le{c\over\nu} |\rot\psi_{xx}|_{3,2,\Omega^t}^2\bigg(D_1^2+{\chi_0^2\over\nu}\bigg)\cr
&\quad+{c\over\nu}|\rot\psi_x|_{3,\infty,\Omega^t}^2(\phi_2+\nu |\nabla\varphi_x(0)|_2^2+|\rot\psi_x(0)|_2^2)\cr
&\quad+{c\over\nu}|\nabla\varphi_{xx}|_{3,2,\Omega^t}^2\bigg(D_1^2+ {\chi_0^2\over\nu}\bigg)+|\nabla\varphi_{xx}(0)|_2^2+{1\over\nu} |\rot\psi_{xx}(0)|_2^2.\cr}
\label{3.38}
\end{equation}
Recall the interpolation
$$\eqal{
&|\rot\psi_{xx}|_{3,2,\Omega^t}^2\bigg(D_1^2+{\chi_0^2\over\nu}\bigg)\le c|\nabla\rot\psi_{xx}|_{2,\Omega^t}^{3/2}|\rot\psi_x|_{2,\Omega^t}^{1/2} \bigg(D_1^2+{\chi_0^2\over\nu}\bigg)\cr
&\le\varepsilon|\nabla\rot\psi_{xx}|_{2,\Omega^t}^2+c(1/\varepsilon) |\rot\psi_x|_{2,\Omega^t}^2\bigg(D_1^2+{\chi_0^2\over\nu}\bigg)^4,\cr}
$$
where in the second term we use the estimate
$$\eqal{
|\rot\psi_x|_{2,\Omega^t}&=|\rot\psi_x+\nabla\varphi_x-\nabla\varphi_x|_{2,\Omega^t} \le c(|v_x|_{2,\Omega^t}+|\nabla\varphi_x|_{2,\Omega^t})\cr
&\le c\bigg(A_1+{\Psi\over\nu}\bigg).\cr}
$$
Next we consider
$$\eqal{
&|\nabla\varphi_{xx}|_{3,2,\Omega^t}^2\bigg(D_1^2+{\chi_0^2\over\nu}\bigg)\le c|\nabla^2\varphi_{xx}|_{2,\Omega^t}^{3/2}|\nabla\varphi_x|_{2,\Omega^t}^{1/2} \bigg(D_1^2+{\chi_0^2\over\nu}\bigg)\cr
&\le\varepsilon|\nabla^2\varphi_{xx}|_{2,\Omega^t}^2+c(1/\varepsilon){\Psi^2\over\nu^2} \bigg(D_1^2+{\chi_0^2\over\nu}\bigg)^4.\cr}
$$
To estimate the second term on the r.h.s. of (\ref{3.38}) we use the interpolation
$$\eqal{
|\rot\psi_x|_{3,\infty,\Omega^t}^2&\le c|\nabla\rot\psi_x|_{2,\infty,\Omega^t}^{3/2}|\rot\psi|_{2,\infty,\Omega^t}^{1/2}\le \varepsilon|\nabla\rot\psi_x|_{2,\infty,\Omega^t}^2\cr
&\quad+c(1/\varepsilon)|\rot\psi|_{2,\infty,\Omega^t}^2\cr}
$$
where
$$
|\rot\psi|_{2,\infty,\Omega^t}=|\rot\psi+\nabla\varphi- \nabla\varphi|_{2,\infty,\Omega^t}\le c\bigg(A_1+{\chi_0\over\sqrt{\nu}}\bigg).
$$
Therefore the second term on the r.h.s. of (\ref{3.38}) is bounded by
$$
{\varepsilon\over\nu}|\nabla\rot\psi_x|_{2,\infty,\Omega^t}^2+{c(1/\varepsilon)\over\nu} \bigg(A_1^2+{\chi_0^2\over\nu}\bigg)(\phi_2+\nu|\nabla\varphi_x(0)|_2^2+ |\rot\psi_x(0)|_2^2)^4.
$$
Using the above estimates in (\ref{3.38}) yields
\begin{equation}\eqal{
&|\nabla\varphi_{xx}(t)|_2^2+{1\over\nu}|\rot\psi_{xx}(t)|_2^2+\mu \bigg(|\nabla^2\varphi_{xx}|_{2,\Omega^t}^2+{1\over\nu} |\nabla\rot\psi_{xx}|_{2,\Omega^t}^2\bigg)\cr
&\quad+\nu|\Delta\varphi_{xx}|_{2,\Omega^t}^2\le{c\over\nu}\bigg(A_1^2+ {\Psi^2\over\nu^2}\bigg)\bigg(D_1^2+{\chi_0^2\over\nu}\bigg)^4\cr
&\quad+{c\over\nu}\bigg(A_1^2+{\chi_0^2\over\nu}\bigg)(\phi_2+\nu |\nabla\varphi_x(0)|_2^2+|\rot\psi_x(0)|_2^2)^4\cr
&\quad+{c\over\nu}{\Psi^2\over\nu^2}\bigg(D_1^2+{\chi_0^2\over\nu}\bigg)^4+ |\nabla\varphi_{xx}(0)|_2^2+{1\over\nu}|\rot\psi_{xx}(0)|_2^2\cr
&\equiv{c\over\nu}\phi_4\bigg(A_1,D_1,{\Psi\over\nu},{\chi_0\over\sqrt{\nu}}, \phi_2+X^2(0)\bigg)+|\nabla\varphi_{xx}(0)|_2^2\cr
&\quad+{1\over\nu}|\rot\psi_{xx}(0)|_2^2\cr}
\label{3.39}
\end{equation}
This inequality implies (\ref{3.32}) and concludes the proof.
\end{proof}

\noindent
To prove a global estimate for solutins to problem (\ref{1.7}) we have to collect all inequalities derived in Lemmas \ref{l3.1}--\ref{l3.5}. This is a topic of the result

\begin{theorem}\label{t3.6}
Assume that $\nabla\varphi(0),\rot\psi(0)\in\Gamma_1^2(\Omega)$ ($\Gamma_1^2(\Omega)$ is defined at the beginning of Section \ref{s2}).\\
Then for $\nu=\nu_*$ sufficiently large and $t\le T<\nu_*^\beta$, $\beta<2(1-\varkappa)$, $\varkappa\in(1/2,1)$, there exists a constant $A=A(\nu_*,X(0)$ (sufficiently large, see (\ref{3.54})) such that
\begin{equation}
X(t)\le A\quad {\rm for}\ t\le T.
\label{3.40}
\end{equation}
We have to emphasize that the bound $A$ in (\ref{3.40}) can be kept the same if $\nu>\nu_*$ is increasing.
\end{theorem}

\begin{proof}
From Lemma \ref{l3.1} we have
\begin{equation}\eqal{
&\nu|\nabla\varphi(t)|_2^2+|\rot\psi(t)|_2^2+\mu(\nu|\nabla^2\varphi|_{2,\Omega^t}^2+ |\nabla\rot\psi|_{2,\Omega^t}^2)\cr
&\quad+\nu^2|\Delta\varphi|_{2,\Omega^t}^2\le cA_1^2D_1^4+\nu|\nabla\varphi(0)|_2^2+|\rot\psi(0)|_2^2.\cr}
\label{3.41}
\end{equation}
Lemma \ref{l3.2} yields
\begin{equation}\eqal{
&\nu|\nabla\varphi_t(t)|_2^2+|\rot\psi_t(t)|_2^2+\mu(\nu |\nabla^2\varphi_t|_{2,\Omega^t}^2+|\nabla\rot\psi_t|_{2,\Omega^t}^2)\cr
&\quad+\nu^2|\Delta\varphi_t|_{2,\Omega^t}^2\le c\phi_1+\nu|\nabla\varphi_t(0)|_2^2+|\rot\psi_t(0)|_2^2,\cr}
\label{3.42}
\end{equation}
where
$$
\phi_1=\bigg(D_1^2+{\chi_0^2\over\nu}\bigg)\bigg(D_2^2+{\Psi^2\over\nu^2}\bigg).
$$
Next Lemma \ref{l3.3} implies
\begin{equation}\eqal{
&\nu|\nabla\varphi_x(t)|_2^2+|\rot\psi_x(t)|_2^2+\mu(\nu |\nabla^2\varphi_x|_{2,\Omega^t}^2+|\nabla\rot\psi_x|_{2,\Omega^t}^2)\cr
&\quad+\nu^2|\Delta\varphi_x|_{2,\Omega^t}^2\le c\phi_2+\nu|\nabla\varphi_x(0)|_2^2+|\rot\psi_x(0)|_2^2,\cr}
\label{3.43}
\end{equation}
where
$$
\phi_2=A_1^2\bigg(D_1^4+D_1^8+{\chi_0^4\over\nu^2}\bigg)+D_1^4\bigg( |\nabla\varphi(0)|_2^2+{1\over\nu}|\rot\psi(0)|_2^2\bigg).
$$
Lemma \ref{l3.4} gives
\begin{equation}\eqal{
&\nu|\nabla\varphi_{xt}(t)|_2^2+|\rot\psi_{xt}(t)|_2^2+\mu(\nu |\nabla^2\varphi_{xt}|_{2,\Omega^t}^2+|\nabla\rot\psi_{xt}|_{2,\Omega^t}^2)\cr
&\quad+\nu^2|\Delta\varphi_{xt}|_{2,\Omega^t}^2\le cD_2^2|\rot\psi_x|_{3,\infty,\Omega^t}^2+c\phi_3\cr
&\quad+\nu|\nabla\varphi_{xt}(0)|_2^2+|\rot\psi_{xt}(0)|_2^2,\cr}
\label{3.44}
\end{equation}
where
$$\eqal{
\phi_3&=\bigg(D_1^4+{\chi_0^4\over\nu^2}\bigg)(\phi_1+\nu |\nabla\varphi_t(0)|_2^2+|\rot\psi_t(0)|_2^2)\cr
&\quad+\bigg(D_2^2+{\chi_0^2\over\nu}\bigg)(\phi_2+\nu|\nabla\varphi_x(0)|_2^2+ |\rot\psi_x(0)|_2^2)+\bigg(D_1^2+{\chi_0^2\over\nu}\bigg){\Psi^2\over\nu^2}.\cr}
$$
Finally, Lemma \ref{l3.5} yields
\begin{equation}\eqal{
&\nu|\nabla\varphi_{xx}(t)|_2^2+|\rot\psi_{xx}(t)|_2^2+\mu(\nu |\nabla^2\varphi_{xx}|_{2,\Omega^t}^2+|\nabla\rot\psi_{xx}|_{2,\Omega^t}^2)\cr
&\quad+\nu^2|\Delta\varphi_{xx}|_{2,\Omega^t}^2\le c\phi_4+\nu|\nabla\varphi_{xx}(0)|_2^2+|\rot\psi_{xx}(0)|_2^2,\cr}
\label{3.45}
\end{equation}
where
$$\eqal{
\phi_4&=\bigg(A_1^2+{\Psi^2\over\nu^2}\bigg)\bigg(D_1^2+{\chi_0^2\over\nu}\bigg)^4+ \bigg(A_1^2+{\chi_0^2\over\nu}\bigg)(\phi_2\cr
&\quad+\nu|\nabla\varphi_x(0)|_2^2+ |\rot\psi_x(0)|_2^2)^4+{\Psi^2\over\nu^2}\bigg(D_1^2+{\chi_0^2\over\nu}\bigg)^4.\cr}
$$
Using (\ref{3.45}) we will be able to estimate the first term on the r.h.s. of (\ref{3.44}). For this purpose we use the interpolation
\begin{equation}\eqal{
&D_2^2|\rot\psi_x|_{3,\infty,\Omega^t}^2\le c|\nabla\rot\psi_x|_{2,\infty,\Omega^t}|\rot\psi_x|_{2,\infty,\Omega^t}D_2^2\cr
&\le\varepsilon|\nabla\rot\psi_x|_{2,\infty,\Omega^t}^2+c/\varepsilon |\rot\psi_x|_{2,\infty,\Omega^t}^2D_2^4,\cr}
\label{3.46}
\end{equation}
where we need to use (\ref{3.43}) to have the estimate
\begin{equation}
|\rot\psi_x|_{2,\infty,\Omega^t}^2\le c(\phi_2+\nu|\nabla\varphi_x(0)|_2^2+ |\rot\psi_x(0)|_2^2).
\label{3.47}
\end{equation}
Summing up (\ref{3.44}) and (\ref{3.45}), using (\ref{3.46}) and (\ref{3.47}) we obtain the inequality
\begin{equation}\eqal{
&\nu(|\nabla\varphi_{xt}(t)|_2^2+|\nabla\varphi_{xx}(t)|_2^2)+|\rot\psi_{xt}(t)|_2^2+ |\rot\psi_{xx}(t)|_2^2\cr
&\quad+\mu[\nu(|\nabla^2\varphi_{xt}|_{2,\Omega^t}^2+ |\nabla^2\varphi_{xx}|_{2,\Omega^t}^2)+|\nabla\rot\psi_{xt}|_{2,\Omega^t}^2\cr
&\quad+|\nabla\rot\psi_{xx}|_{2,\Omega^t}^2]+\nu^2 (|\Delta\varphi_{xt}|_{2,\Omega^t}^2+|\Delta\varphi_{xx}|_{2,\Omega^t}^2)\cr
&\le c(\phi_2+\nu|\nabla\varphi_x(0)|_2^2+|\rot\psi_x(0)|_2^2)D_2^4+c(\phi_3+\phi_4)\cr
&\quad+\nu(|\nabla\varphi_{xt}(0)|_2^2+|\nabla\varphi_{xx}(0)|_2^2)+ |\rot\psi_{xt}(0)|_2^2+|\rot\psi_{xx}(0)|_2^2\cr
&\equiv c\phi_5+\nu(|\nabla\varphi_{xt}(0)|_2^2+|\nabla\varphi_{xx}(0)|_2^2)+ |\rot\psi_{xt}(0)|_2^2+|\rot\psi_{xx}(0)|_2^2.\cr}
\label{3.48}
\end{equation}
Finally, from (\ref{3.41}), (\ref{3.42}), (\ref{3.43}) and (\ref{3.48}) we derive the inequality
\begin{equation}
X^2(t)\le c(A_1^2D_1^4+\phi_1+\phi_2+\phi_5)+X^2(0).
\label{3.49}
\end{equation}
Introducing the notation
$$
B(0)=\nu|\nabla\varphi(0)|_{1,1}^2+|\rot\psi(0)|_{1,1}^2
$$
and using the explicit forms of $\phi_1$, $\phi_2$, $\phi_5$, inequality (\ref{3.49}) takes the form
\begin{equation}\eqal{
X^2(t)&\le c\bigg[\bigg(A_1^2+{\Psi^2\over\nu^2}\bigg)\bigg(D_1^2+{\chi_0^2\over\nu}\bigg)^4\cr
&\quad+ \bigg(1+D_1^4+{\chi_0^4\over\nu^2}\bigg)\bigg(D_1^2+{\chi_0^2\over\nu}\bigg) \bigg(D_2^2+{\Psi^2\over\nu^2}\bigg)\cr
&\quad+A_1^2\bigg(1+D_2^2+D_2^4+{\chi_0^2\over\nu}\bigg)\bigg(D_1^4+D_1^8+ {\chi_0^4\over\nu^2}\bigg)\cr
&\quad+A_1^8\bigg(A_1^2+{\chi_0^2\over\nu}\bigg)\bigg(D_1^4+ D_1^8+{\chi_0^4\over\nu^2}\bigg)^4+\bigg(A_1^2+{\chi_0^2\over\nu}\bigg)(1+D_1^{16})B^4(0)\cr
&\quad+(1+D_1^4)\bigg(D_2^2+ D_2^4+{\chi_0^2\over\nu}+D_1^4+{\chi_0^4\over\nu^2}\bigg)B(0)\cr
&\quad+X^2(0)\bigg]\equiv\phi^2\bigg(D_1,D_2,{\Psi\over\nu},{\chi_0\over\sqrt{\nu}}, B(0)\bigg)+X^2(0).\cr}
\label{3.50}
\end{equation}

\noindent
Using (\ref{2.25}) in (\ref{2.27}) we have the estimate
$$\eqal{
D_1&\le c_1\exp\bigg({X^{2/(1-\varkappa)}\over [c_3-{\sqrt{t}\over\nu^{1-\varkappa}}X]^{1/(1-\varkappa)}}\bigg) \bigg({X\over\nu^{1/3}}+{X^{2/3}\over\nu^{(\varkappa-1/2)/3}}\bigg)+c_1|v(0)|_6+ A_1\cr
&\equiv D_1(X)\cr}
$$
Similarly
$$
D_2=|v_t(0)|_2\exp[(D_1^2A_1^2)/2]\le D_2(X).
$$
Then (\ref{3.50}) implies
\begin{equation}
X(t)\le\phi\bigg(D_1(X),D_2(X),{X\over\nu},{X\over\sqrt{\nu}},B(0)\bigg)+X(0).
\label{3.51}
\end{equation}
Consider the algebraic equation
\begin{equation}
A=\phi\bigg(D_1(A),D_2(A),{A\over\nu},{A\over\sqrt{\nu}},B(0)\bigg)+X(0),
\label{3.52}
\end{equation}
where $\phi\colon\R_+\to\R_+$.

\noindent
If we show the existence of bounded solutions to (\ref{3.52}) such that $A\le A_*$ we get the estimate
\begin{equation}
X\le A.
\label{3.53}
\end{equation}
If we show that $\phi$ is a contraction then the existence of solutions to (\ref{3.52}) follows from the method of successive approximations.

\noindent
Since $\phi$ is a differentiable function of its arguments we restrict our considerations to the case $\phi=\phi(D_1(A))$, because the dependence of the other arguments can be examined similarly. Then we have
$$\eqal{
&|\phi(D_1(A))-\phi(D_1(A'))|\le c|D_1(A)-D_1(A')|\cr
&\le c{|A-A'|\over[c_3-{\sqrt{t}\over\nu^{1-\varkappa}}A]^{1/(1-\varkappa)}} \bigg({A\over\nu^{1/3}} + {A\over\nu^{(\varkappa-1/2)/3}}\bigg)+\,{\rm other\ terms}.\cr}
$$
The first term is bounded by
$$
{c\over\nu^\alpha}|A-A'|,\quad \alpha>0,
$$
where $t\le\nu^\beta$, $\beta<2(1-\varkappa)$. The other terms can be treated similarly. Hence for $\nu$ sufficiently large operator $\phi$ is a contraction.

\noindent
To perform the method of successive approximations we have to find a lower bound for $A$. The lower bound for $A$ must be greater than the r.h.s. of (\ref{3.52}) at $\nu=\infty$. We have
$$\eqal{
&D_1|_{\nu=\infty}=c_1|v(0)|_6+A_1,\cr
&D_2|_{\nu=\infty}=|v_t(0)|_2\exp[(c_1|v(0)|_6+A_1)A_1^2/2]\cr}
$$
Then
\begin{equation}\eqal{
&A>\phi(c_1|v(0)|_2+A_1,|v_t(0)|_2\exp[(c_1|v(0)|_6+A_1)A_1^2/2],\cr
&0,0,B(0))+X(0).\cr}
\label{3.54}
\end{equation}
This concludes the proof.
\end{proof}

\begin{remark}\label{r3.7}
To eliminate the dependence of $t$ in the r.h.s. of (\ref{3.52}) we set $t=\nu^{1-\varkappa}$. Since the dependence appears in $D_1(A)$ only we will eliminate the explicit dependence on time in (\ref{3.51}). Then
$$\eqal{
D_1(A)&=c_1\exp\bigg( {A^{2/(1-\varkappa)}\over[c_3-A/\nu^{(1-\varkappa)/2}]^{1/(1-\varkappa)}}\bigg) \bigg({A\over\nu^{1/3}}+{A^{2/3}\over\nu^{(\varkappa-1/2)/3}}\bigg)\cr
&\quad+c_1|v(0)|_r+A_1.\cr}
$$
Assume that for $\nu=\nu_*$ we found a bound $A_*$ for $A$. We can increase $A$ by $kA$, $k>1$ and also $\nu$ by $k_*\nu$, $k_*>1$, in such a way that the r.h.s. of (\ref{3.52}) does not increase. Then we obtain the same bound $A_*$ but time of existence increases to $t=(k_*\nu_*)^{1-\varkappa}$. Since such $k$ and $k_*$ can be chosen as large as we want the same bound $A_*$ for $A$ holds for $t$ and $\nu$ arbitrary large.
\end{remark}

\section{Global estimate and existence}\label{s4}

In Theorem \ref{t3.6} we proved the existence of long time solutions to problem (\ref{1.7}). However, to prove global existence of solutions to the Navier-Stokes equations (\ref{1.1}) we need to have global existence of solutions to (\ref{1.7}). For this purpose we use the two steps in time technique. We have to emphasize that in this paper any problem of existence is restricted to derive an appropriate estimate. Hence we need the result

\begin{lemma}\label{l4.1}
Let the assumptions of Theorem \ref{t3.6} hold. Let $A$ and $T$ be defined in Theorem \ref{t3.6}. Let
$$
cA^4<{\mu\over 2} T
$$
Let
$$
Y^2(t)=\nu|\nabla\varphi(t)|_{2,1}^2+|\rot\psi(t)|_{2,1}^2
$$
Then
\begin{equation}
Y^2(T)\le\exp\bigg(-{\mu\over 2}T\bigg)Y^2(0).
\label{4.1}
\end{equation}
\end{lemma}

\noindent
This means that the initial data for solutions to problem (\ref{1.7}) for time interval $[T,\infty)$ are small for large $T$. This suggests the existence of global solutions to (\ref{1.7}) in $[T,\infty)$. In Lemma \ref{l4.2} we derive a necessary global estimate.

\begin{proof}
From (\ref{3.3}) and (\ref{3.5}) we have
\begin{equation}\eqal{
&{d\over dt}(\nu|\nabla\varphi|_2^2+|\rot\psi|_2^2)+\mu(\nu|\nabla^2\varphi|_2^2+ |\nabla\rot\psi|_2^2)\cr
&\le c|\rot\psi|_3^2|v|_6^2.\cr}
\label{4.2}
\end{equation}
Next, (\ref{3.9}) and (\ref{3.11}) imply
\begin{equation}\eqal{
&{d\over dt}(\nu|\nabla\varphi_t|_2^2+|\rot\psi_t|_2^2)+\mu(\nu|\nabla^2\varphi_t|_2^2+ |\nabla\rot\psi_t|_2^2)\cr
&\le c(|\rot\psi|_3^2|v_t|_6^2+|\rot\psi_t|_3^2|v|_6^2).\cr}
\label{4.3}
\end{equation}
From (\ref{3.15}) and (\ref{3.19}) it follows
\begin{equation}\eqal{
&{d\over dt}(\nu|\nabla\varphi_x|_2^2+|\rot\psi_x|_2^2)+\mu(\nu|\nabla^2\varphi_x|_2^2 +|\nabla\rot\psi_x|_2^2)\cr
&\le c(|\rot\psi|_3^2|v_x|_6^2+|\rot\psi_x|_3^2|v|_6^2).\cr}
\label{4.4}
\end{equation}
Next, (\ref{3.25}) and (\ref{3.29}) imply
\begin{equation}\eqal{
&{d\over dt}(\nu|\nabla\varphi_{xt}|_2^2+|\rot\psi_{xt}|_2^2)+\mu (\nu|\nabla^2\varphi_{xt}|_2^2+|\nabla\rot\psi_{xt}|_2^2)\cr
&\le(|\rot\psi_{xt}|_3^2|v|_6^2+|\rot\psi_x|_3^2|v_t|_6^2+|\rot\psi_t|_3^2 |v_x|_6^2\cr
&\quad+|\rot\psi|_6^2|v_{xt}|_3^2).\cr}
\label{4.5}
\end{equation}
Finally, (\ref{3.33}) and (\ref{3.36}) yield
\begin{equation}\eqal{
&{d\over dt}(\nu|\nabla\varphi_{xx}|_2^2+|\rot\psi_{xx}|_2^2)+\mu (\nu|\nabla\varphi_{xx}|_2^2+|\nabla\rot\psi_{xx}|_2^2)\cr
&\le c(|\rot\psi_{xx}|_3^2|v|_6^2+|\rot\psi_x|_3^2|v_x|_6^2+|\rot\psi|_6^2 |v_{xx}|_3^2).\cr}
\label{4.6}
\end{equation}
We use the following interpolation
\begin{equation}
|u|_3^2d^2\le c|\nabla u|_2|u|_2d^2\le\varepsilon|\nabla u|_2^2+c/\varepsilon|u|_2^2d^4,
\label{4.7}
\end{equation}
where $d$ is a constant. Using (\ref{4.7}) in (\ref{4.2}) yields
\begin{equation}
{d\over dt}(\nu|\nabla\varphi|_2^2+|\rot\psi|_2^2)+\mu(\nu|\nabla^2\varphi|_2^2+ |\nabla\rot\psi|_2^2)\le c|\rot\psi|_2^2|v|_6^4.
\label{4.8}
\end{equation}
In view of (\ref{4.7}) we obtain from (\ref{4.3}) and (\ref{4.8}) the inequality
\begin{equation}\eqal{
&{d\over dt}(\nu|\nabla\varphi|_2^2+|\rot\psi|_2^2+\nu|\nabla\varphi_t|_2^2+ |\rot\psi_t|_2^2)\cr
&\quad+\mu(\nu|\nabla^2\varphi|_2^2+|\nabla\rot\psi|_2^2+\nu|\nabla^2\varphi_t|_2^2+ |\nabla\rot\psi_t|_2^2)\cr
&\le c(|\rot\psi|_2^2|v|_6^4+|\rot\psi|_2^2|v_t|_6^4+|\rot\psi_t|_2^2|v|_6^4).\cr}
\label{4.9}
\end{equation}
In view of (\ref{4.7}) inequalities (\ref{4.4}) and (\ref{4.8}) imply
\begin{equation}\eqal{
&{d\over dt}(\nu|\nabla\varphi|_2^2+|\rot\psi|_2^2+\nu|\nabla\varphi_x|_2^2+ |\rot\psi_x|_2^2)\cr
&\quad+\mu(\nu|\nabla^2\varphi|_2^2+|\nabla\rot\psi|_2^2+\nu|\nabla^2\varphi_x|_2^2+ |\nabla\rot\psi_x|_2^2)\cr
&\le c(|\rot\psi|_2^2|v|_6^4+|\rot\psi|_2^2|v_t|_6^4+|\rot\psi|_2^2|v_x|_6^4+ |\rot\psi_x|_2^2|v|_6^4+|\rot\psi_t|_2^2|v|_6^4).\cr}
\label{4.10}
\end{equation}
Using (\ref{4.7}) inequalities (\ref{4.5}), (\ref{4.9}) and (\ref{4.10}) give
\begin{equation}\eqal{
&{d\over dt}(\nu|\nabla\varphi|_2^2+|\rot\psi|_2^2+\nu|\nabla\varphi_t|_2^2+ |\rot\psi_t|_2^2+\nu|\nabla\varphi_x|_2^2+|\rot\psi_x|_2^2\cr
&\quad+\nu|\nabla\varphi_{xt}|_2^2+|\rot\psi_{xt}|_2^2)+\mu(\nu|\nabla^2\varphi|_2^2+ |\nabla\rot\psi|_2^2+\nu|\nabla^2\varphi_t|_2^2\cr
&\quad+|\nabla\rot\psi_t|_2^2+\nu|\nabla^2\varphi_x|_2^2+|\nabla\rot\psi_x|_2^2+ \nu|\nabla^2\varphi_{xt}|_2^2+|\nabla\rot\psi_{xt}|_2^2)\cr
&\le c(|\rot\psi_{xt}|_2^2|v|_6^4+|\rot\psi_x|_2^2|v_t|_6^4+|\rot\psi_t|_2^2|v_x|_6^4\cr
&\quad+|\rot\psi_t|_2^2|v|_6^4+|\rot\psi|_2^2|v|_6^4+|\rot\psi|_2^2|v_t|_6^4+ |\rot\psi|_2^2|v_x|_6^2\cr
&\quad+|v_{xt}|_2^2|\rot\psi|_6^4).\cr}
\label{4.11}
\end{equation}
Finally, (\ref{4.2}), (\ref{4.4}) and (\ref{4.6}) imply
\begin{equation}\eqal{
&{d\over dt}(\nu|\nabla\varphi|_2^2+|\rot\psi|_2^2+\nu|\nabla\varphi_x|_2^2+ |\rot\psi_x|_2^2+\nu|\nabla\varphi_{xx}|_2^2\cr
&\quad+|\rot\psi_{xx}|_2^2)+\mu(\nu|\nabla^2\varphi|_2^2+|\nabla\rot\psi|_2^2+ \nu|\nabla^2\varphi_x|_2^2\cr
&\quad+|\nabla\rot\psi_x|_2^2+\nu|\nabla^2\varphi_{xx}|_2^2+|\nabla\rot\psi_{xx}|_2^2)\cr
&\le c(|\rot\psi_{xx}|_2^2|v|_6^4+|\rot\psi_x|_2^2|v_x|_6^4+|\rot\psi|_2^2|v|_6^4\cr
&\quad+ |\rot\psi|_2^2|v_x|_6^4+|\rot\psi_x|_2^2|v|_6^4+|v_{xx}|_2^2|\rot\psi|_6^4).\cr}
\label{4.12}
\end{equation}
Introduce the quantity
$$
Y^2(t)=\nu|\nabla\varphi(t)|_{2,1}^2+|\rot\psi(t)|_{2,1}^2.
$$
Then inequalities (\ref{4.8})--(\ref{4.12}) imply the inequality
\begin{equation}
{d\over dt}Y^2+\mu Y^2\le cY^2(|v|_6^4+|v_x|_6^4+|v_t|_6^4+|\rot\psi|_6^4).
\label{4.13}
\end{equation}
From (\ref{4.13}) we have
\begin{equation}
{d\over dt}\bigg(Y^2\exp\bigg(\mu t-c|v|_{2,1,\infty,\Omega^t}^2\intop_0^t|v(t')|_{2,1}^2dt'\bigg)\bigg)\le 0.
\label{4.14}
\end{equation}
Hence the decay follows
\begin{equation}
Y^2(t)\le\exp(-\mu t+cA^4)Y^2(0).
\label{4.15}
\end{equation}
For $t=T$, the restriction $cA^4\le{\mu\over 2}T$ and Remark \ref{r3.7} we derive (\ref{4.1}) and conclude the proof.
\end{proof}

\noindent
In Theorem \ref{t3.6} estimate (\ref{3.40}) is proved for $t\le T$, where $T<\nu^{2(1-\varkappa)}$. To prove Theorem \ref{t3.6} we need that $\nu$ is large so $T$ can also be large. To show (\ref{3.40}) we need that the norm $X(0)=Y(0)$ of initial data for solutions to (\ref{1.7}) must be finite. Lemma \ref{l4.1} implies that the norm $Y(T)$ for $T$ sufficiently large satisfies
\begin{equation}
Y(T)\le Y(0).
\label{4.16}
\end{equation}
It suggests that the data at time $T$ to problem (\ref{1.7}) satisfying (\ref{4.16}) can be treated as the initial data for $[T,2T]$ so also the proof of Theorem \ref{t3.6} can be repeated for this interval. However, to prove Theorem \ref{t3.6} for interval $[T,2T]$ we need Lemma \ref{l2.4}. Therefore the following assumptions are necessary
\begin{equation}
{c_3\over\nu^\varkappa}\le|\varphi(T)|_p\le{c_4\over\nu^\varkappa},\quad |\varphi(T)|_2\le{c_4\over\nu^\varkappa}.
\label{4.17}
\end{equation}
But we do not know how satisfy (\ref{4.17}). Therefore we have to derive an estimate of type (\ref{3.40}) for $t>T$ in a different way. This is the topic of the next lemma

\begin{lemma}\label{l4.2}
Assume that $T$ and $\nu$ are large. Assume that $\varphi(T)\in H^3(\Omega)$, $v(T)\in H^2(\Omega)$. Then there exists a constant $B$ such that
$$
\|\varphi(T)\|_3+\|v(T)\|_2<cB,
$$
where $B=\big(e^{-{\mu\over2}T}\|v(0)\|_2+{A\over\sqrt{\nu}}\big)$ and $A$ appears in (\ref{3.40}). Moreover, for $T$ and $\nu$ sufficiently large the following inequality holds
\begin{equation}
\|v\|_{2,\infty,\Omega_T^t}+\|v\|_{3,2,\Omega_T^t}+\|\varphi\|_{3,\infty,\Omega_T^t}+ \|\nabla\varphi\|_{3,2,\Omega_T^t}<cB
\label{4.18}
\end{equation}
for any $t\in(T,\infty)$.
\end{lemma}

\begin{proof}
The first statement follows directly from (\ref{4.15}). To derive the estimate we first need to show that data for solutions of (\ref{1.7}) at time $T$ are sufficiently small for sufficiently large $T$. Then we get problem (\ref{1.7}) with small initial data so it is clear that we would be able to show (\ref{3.40}) for any $t>T$.

\noindent
Multiply (\ref{1.7}) by $v$ and integrate over $\Omega$. Then we have
\begin{equation}
{d\over dt}|v|_2^2+\mu|\nabla v|_2^2+\nu|\divv v|_2^2=0.
\label{4.19}
\end{equation}
Since $\intop_\Omega vdx=0$, (\ref{4.19}) yields
$$
{d\over dt}(|v|_2^2e^{\mu t})\le 0
$$
so
\begin{equation}
|v(t)|_2^2\le e^{-\mu t}|v(0)|_2^2.
\label{4.20}
\end{equation}
Let $T$ be large. Then the quantity $|v(T)|_2$
\begin{equation}
|v(T)|_2^2\le e^{-\mu T}|v(0)|_2^2
\label{4.21}
\end{equation}
is small. Integrating (\ref{4.19}) with respect to time from $t=T$ to $t>T$ we get
\begin{equation}
|v(t)|_2^2+\mu|\nabla v|_{2,\Omega_T^t}^2+\nu|\Delta\varphi|_{2,\Omega_T^t}^2= |v(T)|_2^2.
\label{4.22}
\end{equation}
Since
\begin{equation}
\intop_\Omega vdx=0,\quad \intop_\Omega\varphi dx=0
\label{4.23}
\end{equation}
we obtain from (\ref{4.22}) the inequality
\begin{equation}
|v(t)|_2^2+\mu\|v\|_{1,2,\Omega_T^t}^2+\nu\|\varphi\|_{2,2,\Omega_T^t}^2\le c|v(T)|_2^2.
\label{4.24}
\end{equation}
Differentiate $(\ref{1.7})_1$ with respect to $x$, multiply by $v_x$ and integrate over $\Omega$. Then we have
$$\eqal{
&{d\over dt}|v_x|_2^2+\mu|\nabla v_x|_2^2+\nu|\divv v_x|_2^2=-\intop_\Omega (\rot\psi)_x\cdot\nabla v\cdot v_xdx\cr
&=\intop_\Omega(\rot\psi)_x\cdot\nabla v_x\cdot vdx.\cr}
$$
Applying the H\"older and Young inequalities to the r.h.s., integrating with respect to time from $t=T$ to $t>T$ and using (\ref{4.24}), we obtain
\begin{equation}\eqal{
&\|v(t)\|_1^2+\mu\|v\|_{2,2,\Omega_T^t}^2+\nu\|\Delta\varphi\|_{1,2,\Omega_T^t}^2\cr
&\le c|\rot\psi_x|_{3,\infty\Omega_T^t}^2|v(T)|_2^2+c\|v(T)\|_1^2.\cr}
\label{4.25}
\end{equation}
Differentiate $(\ref{1.7})_1$ twice with respect to $x$, multiply by $v_{xx}$ and integrate over $\Omega$. Then we derive
\begin{equation}\eqal{
&{d\over dt}|v_{xx}|_2^2+\mu|\nabla v_{xx}|_2^2+\nu|\divv v_{xx}|_2^2=-\intop_\Omega(\rot\psi)_{xx}\cdot\nabla v\cdot v_{xx}dx\cr
&\quad-2\intop_\Omega(\rot\psi)_x\cdot\nabla v_x\cdot v_{xx}dx-\intop_\Omega\rot\psi\cdot\nabla v_{xx}\cdot v_{xx}dx\cr
&\equiv I_1+I_2+I_3,\cr}
\label{4.26}
\end{equation}
where by the H\"older and Young inequalities we get
$$\eqal{
&|I_1|\le\varepsilon|v_{xx}|_6^2+c/\varepsilon|\nabla v|_2^2|\rot\psi_{xx}|_3^2,\cr
&|I_2|\le\varepsilon|v_{xx}|_6^2+c/\varepsilon|\rot\psi_x|_2^2|v_{xx}|_3^2,\cr
&I_3=0.\cr}
$$
Using the above estimates in (\ref{4.26}), summing up the result integrated with respect to time with (\ref{4.25}) and assuming that $\varepsilon$ is sufficiently small we derive the inequality
\begin{equation}\eqal{
&\|v(t)\|_2^2+\mu\|v\|_{3,2,\Omega_T^t}^2+\nu\|\Delta\varphi\|_{2,2,\Omega_T^t}^2\cr
&\le c|\rot\psi_x|_{3,\infty,\Omega_T^t}^2|v(T)|_2^2+c|\nabla v|_{2,\infty,\Omega_T^t}^2|\rot\psi_{xx}|_{3,2,\Omega_T^t}^2\cr
&\quad+c|\rot\psi_x|_{2,\infty,\Omega_T^t}^2|v_{xx}|_{3,2,\Omega_T^t}^2+ c\|v(T)\|_2^2.\cr}
\label{4.27}
\end{equation}
Using the interpolations
$$\eqal{
&|v(T)|_2^2|\rot\psi_x|_{3,\infty,\Omega_T^t}^2\le c|v(T)|_2^2(\varepsilon^{1/2} |\rot\psi_{xx}|_{2,\infty,\Omega_T^t}^2\cr
&\quad+c\varepsilon^{-3/2}|\rot\psi|_{2,\infty,\Omega_T^t}^2)=\varepsilon_1^{1/2} |\rot\psi_{xx}|_{2,\infty,\Omega_T^t}^2\cr
&\quad+c\varepsilon_1^{-3/2}|v(T)|_2^8|\rot\psi|_{2,\infty,\Omega_T^t}^2,\cr}
$$
$$\eqal{
&|\nabla v|_{2,\infty,\Omega_T^t}^2|\rot\psi_{xx}|_{3,2,\Omega_T^t}^2\le|\nabla v|_{2,\infty,\Omega_T^t}^2(\varepsilon^{1/3}|\rot\psi_{xxx}|_{2,\Omega_T^t}^2\cr
&\quad+c\varepsilon^{-5/3}|\rot\psi|_{2,\Omega_T^t}^2)=\varepsilon_2^{1/3} |\rot\psi_{xxx}|_{2,\Omega_T^t}^2\cr
&\quad+c\varepsilon_2^{-5/3}|\nabla v|_{2,\infty,\Omega_T^t}^{12}|\rot\psi|_{2,\Omega_T^t}^2\cr}
$$
and
$$\eqal{
&|\rot\psi|_{2,\infty,\Omega_T^t}^2\le|v|_{2,\infty,\Omega_T^t}^2\le|v(T)|_2^2,\cr
&|\rot\psi|_{2,\Omega_T^t}^2\le|v|_{2,\Omega_T^t}^2\le|v(T)|_2^2,\cr}
$$
and similarly we have
$$%\eqal{&
|\rot\psi_x|_{2,\infty,\Omega_T^t}^2|v_{xx}|_{3,2,\Omega_T^t}^2\le \varepsilon_3^{1/3}|v_{xxx}|_{2,\Omega_T^t}^2%\cr&\quad
+c\varepsilon_3^{-5/3}|\rot\psi_x|_{2,\infty,\Omega_T^t}^{12} |v|_{2,\Omega_T^t}^2.%\cr}
$$
Using $\rot\psi=v-\nabla\varphi$ and
$$
|v(t)|_2^2=e^{-\mu(t-T)}|v(T)|_2^2
$$
so
\begin{equation}
\intop_T^t|v(t')|_2^2\le|v(T)|_2^2
\label{4.28}
\end{equation}
we obtain from (\ref{4.27}) the following inequality for sufficiently small $\varepsilon_1-\varepsilon_3$
\begin{equation}\eqal{
&\|v(t)\|_2^2+\mu\|v\|_{3,2,\Omega_T^t}^2+\nu\|\Delta\varphi\|_{2,2,\Omega_T^t}^2%\cr&
\le c|\nabla v|_{2,\infty,\Omega_T^t}^{12}|v(T)|_2^2\cr
&\quad+c|v(T)|_2^{12}|v(T)|_2^2+ c|\rot\psi_x|_{2,\infty,\Omega_T^t}^{12}|v(T)|_2^2+c\|v(T)\|_2^2.\cr}
\label{4.29}
\end{equation}
Introduce the quantity
$$
X_0(t)=\|v(t)\|_2^2+\mu\|v\|_{3,2,\Omega_T^t}^2+\nu\|\Delta\varphi\|_{2,2,\Omega_T^t}^2.
$$
Then (\ref{4.29}) implies for $t>T$
\begin{equation}
X_0(t)\le cX_0^6(t)|v(T)|_2^2+c\|v(T)\|_2^2.
\label{4.30}
\end{equation}
Then for $T$ sufficiently large and a fixed point argument we obtain the estimate
\begin{equation}
X_0(t)\le c_0\|v(T)\|_2^2,
\label{4.31}
\end{equation}
where $c_0$ is a correspondingly large constant.

\noindent
To derive more delicate estimate for $\varphi$ we consider problem (\ref{1.7}). From (\ref{1.7}) we have
\begin{equation}
\varphi_t-(\mu+\nu)\Delta\varphi=\Delta^{-1}\partial_{x_i}\partial_{x_j}(v_iv_j)- \Delta^{-1}\partial_{x_i}\partial_{x_j}(\varphi_{x_i}v_j).
\label{4.32}
\end{equation}
Multiplying (\ref{4.32}) by $\varphi$ and integrating over $\Omega$ yields
\begin{equation}\eqal{
&{d\over dt}|\varphi|_2^2+(\mu+\nu)|\nabla\varphi|_2^2\le{c\over\nu}|vv|_{6/5}^2+ {c\over\nu}|\nabla\varphi v|_{6/5}^2\cr
&\le{c\over\nu}(|v|_2^2|v|_3^2+|v|_2^2|\nabla\varphi|_3^2).\cr}
\label{4.33}
\end{equation}
Integrating (\ref{4.33}) with respect to time implies
\begin{equation}\eqal{
&|\varphi(t)|_2^2+(\mu+\nu)|\nabla\varphi|_{2,\Omega_T^t}^2\le{c\over\nu} |v(T)|_2^2(|v|_{3,2,\Omega_T^t}^2\cr
&\quad+|\nabla\varphi|_{3,2,\Omega_T^t}^2)+|\varphi(T)|_2^2.\cr}
\label{4.34}
\end{equation}
Multiplying $(\ref{1.7})_1$ by $\nabla\varphi$ and integrating over $\Omega$ gives
$$
{d\over dt}|\nabla\varphi|_2^2+\mu|\nabla^2\varphi|_2^2+\nu|\Delta\varphi|_2^2\le {c\over\nu}|\rot\psi|_3^2|v|_6^2.
$$
Integrating with respect to time implies
\begin{equation}
|\nabla\varphi(t)|_2^2+\mu|\nabla^2\varphi|_{2,\Omega_T^t}^2+\nu |\Delta\varphi|_{2,\Omega_T^t}^2\le{c\over\nu}|\rot\psi|_{3,\infty,\Omega_T^t} |v|_{6,2,\Omega_T^t}^2+|\nabla\varphi(T)|_2^2.
\label{4.35}
\end{equation}
Differentiate $(\ref{1.7})_1$ with respect to $x$, multiply by $\nabla\varphi_x$ and integrate over $\Omega$. Then we obtain
\begin{equation}\eqal{
&{d\over dt}|\nabla\varphi_x|_2^2+\mu|\nabla^2\varphi_x|_2^2+\nu|\Delta\varphi|_2^2= -\intop_\Omega\rot\psi\cdot\nabla v_x\cdot\nabla\varphi_xdx\cr
&\quad-\intop_\Omega\rot\psi_x\cdot\nabla v\cdot\nabla\varphi_xdx=\intop_\Omega\rot\psi\cdot\nabla\nabla\varphi_x\cdot v_xdx\cr
&\quad+\intop_\Omega\rot\psi_x\cdot\nabla\nabla\varphi_x\cdot vdx\equiv I_1+I_2,\cr}
\label{4.36}
\end{equation}
where
$$\eqal{
&|I_1|\le\varepsilon|\nabla^2\varphi_x|_2^2+c/\varepsilon|\rot\psi|_6^2|v_x|_3^2,\cr
&|I_2|\le\varepsilon|\nabla^2\varphi_x|_2^2+c/\varepsilon|\rot\psi_x|_3^2|v|_6^2.\cr}
$$
Using the above estimates in (\ref{4.36}), assuming that $\varepsilon$ is sufficiently small and integrating the result with respect to time we have
\begin{equation}\eqal{
&|\nabla\varphi_x(t)|_2^2+\mu|\nabla^2\varphi_x|_{2,\Omega_T^t}^2+ \nu|\Delta\varphi_x|_{2,\Omega_T^t}^2\cr
&\le{c\over\nu}|\rot\psi|_{6,\infty,\Omega_T^t}^2|v_x|_{3,2,\Omega_T^t}^2+ {c\over\nu}|\rot\psi_x|_{3,2,\Omega_T^t}^2|v|_{6,\infty,\Omega_T^t}^2\cr
&\quad+|\nabla\varphi_x(T)|_2^2.\cr}
\label{4.37}
\end{equation}
Differentiate $(\ref{1.7})_1$ twice with respect to $x$, multiply by $\nabla\varphi_{xx}$ and integrate over $\Omega$. Then we obtain
\begin{equation}\eqal{
&{d\over dt}|\nabla\varphi_{xx}|_2^2+\mu|\nabla^2\varphi_{xx}|_2^2+\nu|\Delta\varphi_{xx}|_2^2\cr
&\le\bigg|-\intop_\Omega(\rot\psi\cdot\nabla v)_{xx}\cdot\nabla\varphi_{xx}dx\bigg|=\bigg|\intop_\Omega(\rot\psi\cdot\nabla v)_x\cdot\nabla\varphi_{xxx}dx\bigg|\cr
&\le\varepsilon|\nabla\varphi_{xxx}|_2^2+{c\over\varepsilon}|\rot\psi|_6^2|\nabla v_x|_3^2+{c\over\varepsilon}|\rot\psi_x|_3^2|\nabla v|_6^2.\cr}
\label{4.38}
\end{equation}
Integrating (\ref{4.38}) with respect to time yields
\begin{equation}\eqal{
&|\nabla\varphi_{xx}(t)|_2^2+\mu|\nabla^2\varphi_{xx}|_{2,\Omega_T^t}^2+\nu |\Delta\varphi_{xx}|_{2,\Omega_T^t}^2\cr
&\le{c\over\nu}|\rot\psi|_{6,\infty,\Omega_T^t}^2|\nabla v_x|_{3,2,\Omega_T^t}^2\cr
&\quad+{c\over\nu}|\rot\psi_x|_{3,\infty,\Omega_T^t}^2|\nabla v|_{6,2,\Omega_T^t}^2+|\nabla\varphi_{xx}(T)|_2^2.\cr}
\label{4.39}
\end{equation}
From (\ref{4.34}), (\ref{4.35}), (\ref{4.37}) and (\ref{4.39}) it follows
\begin{equation}\eqal{
&\|\varphi(t)\|_3^2+\mu\|\nabla\varphi\|_{3,2,\Omega_T^t}^2+ \nu\|\Delta\varphi\|_{2,2,\Omega_T^t}^2\cr
&\le{c\over\nu}|v(T)|_2^2(|v|_{3,2,\Omega_T^t}^2+ |\nabla\varphi|_{3,2,\Omega_T^t}^2)\cr
&\quad+{c\over\nu}(\|\rot\psi\|_{2,\infty,\Omega_T^t}^2+ \|\rot\psi\|_{2,2,\Omega_T^t}^2)\cdot\cr
&\quad\cdot(\|v\|_{3,2,\Omega_T^t}^2+|v|_{6,\infty,\Omega_T^t}^2)+ \|\varphi(T)\|_3^2.\cr
&\le{c\over\nu}\|v(T)\|_2^4+{c\over\nu}\|\varphi(T)\|_3^2,\cr}
\label{4.40}
\end{equation}
where the last inequality follows from (\ref{4.31}). Hence (\ref{4.31}) and (\ref{4.40}) imply (\ref{4.18}). This concludes the proof.
\end{proof}

\section{The energy type estimates for $u$}\label{s5}

\begin{lemma}\label{l5.1}
Let the assumptions of Theorem \ref{t3.6} hold. Let $u(t)\in H^1(\Omega)$, $u_t\in L_2(\Omega)$. Then
\begin{equation}\eqal{
&{d\over dt}|u|_2^2+\mu\|u\|_1^2\le c|u|_2^2(|\nabla v|_6^2+|\nabla\varphi|_6^4)\cr
&\quad+c(\|v\|_1^2|\nabla\varphi|_6^2+\|\nabla\varphi\|_1^2+|\nabla\varphi_t|_2^2).\cr}
\label{5.1}
\end{equation}
\end{lemma}

\begin{proof}
Since $u$ is not divergence free we multiply (\ref{1.9}) by $u-\nabla\varphi$. Integrating the result over $\Omega$ we obtain
\begin{equation}\eqal{
&\intop_\Omega u_t\cdot(u-\nabla\varphi)dx+\intop_\Omega\rot\psi\cdot\nabla u\cdot(u-\nabla\varphi)dx\cr
&\quad+\intop_\Omega(u-\nabla\varphi)\cdot\nabla(v-u)\cdot(u-\nabla\varphi)dx+\mu \intop_\Omega\nabla u\cdot\nabla(u-\nabla\varphi)dx=0.\cr}
\label{5.2}
\end{equation}
Continuing, we have
\begin{equation}\eqal{
&{1\over2}{d\over dt}|u|_2^2-\intop_\Omega u_t\cdot\nabla\varphi dx+\intop_\Omega\rot\psi\cdot\nabla u\cdot udx\cr
&\quad-\intop_\Omega\rot\psi\cdot\nabla u\cdot\nabla\varphi dx-\intop_\Omega(u-\nabla\varphi)\cdot\nabla u\cdot udx\cr
&\quad+\intop_\Omega(u-\nabla\varphi)\cdot\nabla u\cdot\nabla\varphi dx+\intop_\Omega(u-\nabla\varphi)\cdot\nabla v\cdot(u-\nabla\varphi)dx\cr
&\quad+\mu|\nabla u|_2^2-\mu\intop_\Omega\nabla u\cdot\nabla^2\varphi dx=0.\cr}
\label{5.3}
\end{equation}
The second term on the l.h.s. of (\ref{5.3}) equals
$$
\intop_\Omega\divv u_t\varphi dx=\intop_\Omega\Delta\varphi_t\varphi dx=-\intop_\Omega\nabla\varphi_t\cdot\nabla\varphi dx=-{1\over2}{d\over dt}|\nabla\varphi|_2^2.
$$
The third term vanishes. We estimate the fourth term by
$$
\varepsilon|\nabla u|_2^2+c/\varepsilon|\rot\psi|_3^2|\nabla\varphi|_6^2.
$$
The fifth term vanishes. We estimate the sixth term by
$$
\varepsilon|\nabla u|_2^2+c/\varepsilon(|u|_3^2+|\nabla\varphi|_3^2)|\nabla\varphi|_6^2.
$$
The seventh term is written in the form
$$\eqal{
&\intop_\Omega u\cdot\nabla vudx-\intop_\Omega\nabla\varphi\cdot\nabla v\cdot udx-\intop_\Omega u\cdot\nabla v\cdot\nabla\varphi dx+\intop_\Omega\nabla\varphi\cdot\nabla v\cdot\nabla\varphi dx\cr
&\equiv J_1+J_2+J_3+J_4,\cr}
$$
where
$$\eqal{
&|J_1|\le\varepsilon|u|_6^2+c/\varepsilon|u|_2^2|\nabla v|_6^2,\cr
&|J_2|\le\varepsilon|u|_6^2+c/\varepsilon|\nabla v|_2^2|\nabla\varphi|_3^2,\cr
&|J_3|\le\varepsilon|u|_6^2+c/\varepsilon|\nabla v|_2^2|\nabla\varphi|_3^2,\cr
&|J_4|\le|\nabla\varphi|_6^2+|\nabla v|_2^2|\nabla\varphi|_3^2.\cr}
$$
Finally, the last term is bounded by
$$
\varepsilon|\nabla u|_2^2+c/\varepsilon|\nabla^2\varphi|_2^2.
$$
Employing the above estimates in (\ref{5.3}) and assuming that $\varepsilon$ is sufficiently small yield
\begin{equation}\eqal{
&{d\over dt}|u|_2^2+\mu\|u\|_1^2\le{d\over dt}|\nabla\varphi|_2^2+c(|\rot\psi|_3^2|\nabla\varphi|_6^2\cr
&\quad+|u|_3^2|\nabla\varphi|_6^2+|\nabla\varphi|_3^2|\nabla\varphi|_6^2+|u|_2^2 |\nabla v|_6^2+|\nabla v|_2^2|\nabla\varphi|_3^2\cr
&\quad+|\nabla^2\varphi|_2^2+|\nabla\varphi|_6^2+|\nabla\varphi|_2^2).\cr}
\label{5.4}
\end{equation}
By some interpolation inequality we have
$$
|u|_3^2|\nabla\varphi|_6^2\le\varepsilon|\nabla u|_2^2+c/\varepsilon|u|_2^2|\nabla\varphi|_6^4.
$$
Then we write (\ref{5.4}) in the form
\begin{equation}\eqal{
&{d\over dt}|u|_2^2+\mu\|u\|_1^2\le c|u|_2^2(|\nabla v|_6^2+|\nabla\varphi|_6^4)\cr
&\quad+c[(|\rot\psi|_3^2+|\nabla\varphi|_3^2)|\nabla\varphi|_6^2+|\nabla v|_2^2|\nabla\varphi|_3^2\cr
&\quad+|\nabla^2\varphi|_2^2+|\nabla\varphi|_6^2]+{d\over dt}|\nabla\varphi|_2^2.\cr}
\label{5.5}
\end{equation}
Using that
$$
|\rot\psi|_3^2+|\nabla\varphi|_3^2\le c(\|\rot\psi\|_1^2+\|\nabla\varphi\|_1^2)\le c\|v\|_1^2
$$
we simplify (\ref{5.5}) to the inequality
\begin{equation}\eqal{
&{d\over dt}|u|_2^2+\mu\|u\|_1^2\le c|u|_2^2(|\nabla v|_6^2+|\nabla\varphi|_6^4)\cr
&\quad+c(\|v\|_1^2|\nabla\varphi|_6^2+|\nabla^2\varphi|_2^2+ |\nabla\varphi|_6^2)+{d\over dt}|\nabla\varphi|_2^2.\cr}
\label{5.6}
\end{equation}
This implies (\ref{5.1}) and concludes the proof.
\end{proof}

\noindent
Next we have

\begin{lemma}\label{l5.2}
Let the assumptions of Theorem \ref{t3.6} hold, let $u(t)\in H^2(\Omega)$, $u_t\in H^1(\Omega)$. Then
\begin{equation}\eqal{
&{d\over dt}|u_x|_2^2+\mu\|\nabla u\|_1^2\le c|u_x|_2^6+c\|u\|_1^2\|v\|_2^2\cr
&\quad+c(\|v\|_2^2\|\nabla\varphi\|_1^2+|\nabla\varphi_{xt}|_2^2+ \|\nabla\varphi\|_2^2+|\nabla\varphi|_6^2|\nabla\varphi_x|_3^2)\cr
&\quad+|\nabla\varphi_x|_2^2|\nabla\varphi_x|_3^2).\cr}
\label{5.7}
\end{equation}
\end{lemma}

\begin{proof}
Differentiate (\ref{1.9}) with respect to $x$, multiply by $(u-\nabla\varphi)_x$ and integrate the result over $\Omega$. Then we have
\begin{equation}\eqal{
&\intop_\Omega u_{xt}\cdot(u_x-\nabla\varphi_x)dx+\intop_\Omega(\rot\psi_x\cdot \nabla u+\rot\psi\cdot\nabla u_x)(u_x-\nabla\varphi_x)dx\cr
&\quad+\intop_\Omega[(u-\nabla\varphi)\cdot\nabla(v-u)]_{,x}\cdot (u_x-\nabla\varphi_x)dx\cr
&\quad+\mu\intop_\Omega\nabla u_x\cdot(\nabla u_x-\nabla^2\varphi_x)dx=0.\cr}
\label{5.8}
\end{equation}
The first term on the l.h.s. of (\ref{5.8}) equals
$$
{1\over2}{d\over dt}|u_x|_2^2-\intop_\Omega u_{xt}\cdot\nabla\varphi_xdx,
$$
where the second term reads
$$
-\intop_\Omega\nabla\varphi_{xt}\cdot\nabla\varphi_xdx=-{1\over2}{d\over dt}|\nabla\varphi_x|_2^2.
$$
We write the second term on the l.h.s. of (\ref{5.8}) in the form
$$\eqal{
&\intop_\Omega\rot\psi_x\cdot\nabla u\cdot u_xdx-\intop_\Omega\rot\psi_x\cdot\nabla u\cdot\nabla\varphi_xdx+\intop_\Omega\rot\psi\cdot\nabla u_x\cdot u_xdx\cr
&\quad-\intop_\Omega\rot\psi\cdot\nabla u_x\cdot\nabla\varphi_xdx\equiv\sum_{i=1}^4I_i,\cr}
$$
where
$$\eqal{
&|I_1|\le\varepsilon|u_x|_6^2+c/\varepsilon|\rot\psi_x|_3^2|\nabla u|_2^2,\cr
&|I_2|\le\varepsilon|\nabla u|_6^2+c/\varepsilon|\rot\psi_x|_3^2|\nabla\varphi_x|_2^2,\cr
&I_3=0\quad {\rm and}\cr
&|I_4|\le\varepsilon|\nabla u_x|_2^2+c/\varepsilon|\rot\psi|_\infty^2|\nabla\varphi_x|_2^2.\cr}
$$
Next we examine the third term on the l.h.s. of (\ref{5.8}). We write it in the form
$$\eqal{
&\intop_\Omega(u_x-\nabla\varphi_x)\cdot\nabla(v-u)\cdot(u_x-\nabla\varphi_x)dx\cr
&\quad+ \intop_\Omega(u-\nabla\varphi)\cdot\nabla(v_x-u_x)\cdot(u_x-\nabla\varphi_x)dx\equiv J+L.\cr}
$$
First we consider $J$. It can be expressed in the form
$$\eqal{
&-\intop_\Omega(u_x-\nabla\varphi_x)\cdot\nabla u\cdot(u_x-\nabla\varphi_x)dx\cr
&\quad+ \intop_\Omega(u_x-\nabla\varphi_x)\cdot\nabla v\cdot(u_x-\nabla\varphi_x)dx\cr
&=-\intop_\Omega u_x\cdot\nabla u\cdot u_xdx+\intop_\Omega u_x\cdot\nabla u\cdot\nabla\varphi_xdx+\intop_\Omega\nabla\varphi_x\cdot\nabla u\cdot u_xdx\cr
&\quad-\intop_\Omega\nabla\varphi_x\cdot\nabla u\cdot\nabla\varphi_xdx+\intop_\Omega u_x\cdot\nabla v\cdot u_xdx-\intop_\Omega u_x\cdot\nabla v\cdot\nabla\varphi_xdx\cr
&\quad-\intop_\Omega\nabla\varphi_x\cdot\nabla v\cdot u_xdx+\intop_\Omega\nabla\varphi_x\cdot\nabla v\cdot\nabla\varphi_xdx\equiv\sum_{i=1}^8J_i,\cr}
$$
where
$$\eqal{
&|J_1|\le|u_x|_3^3\le c|u_{xx}|_2^{3/2}|u_x|_2^{3/2}\le\varepsilon|u_{xx}|_2^2+ c(1/\varepsilon)|u_x|_2^6,\cr
&|J_2|\le\varepsilon|u_x|_6^2+c/\varepsilon|\nabla u|_2^2|\nabla\varphi_x|_3^2,\cr
&|J_3|\le\varepsilon|u_x|_6^2+c/\varepsilon|\nabla u|_2^2|\nabla\varphi_x|_3^2,\cr
&|J_4|\le\varepsilon|\nabla u|_6^2+c/\varepsilon|\nabla\varphi_x|_2^2 |\nabla\varphi_x|_3^2,\cr
&|J_5|\le\varepsilon|u_x|_6^2+c/\varepsilon|u_x|_2^2|\nabla v|_3^2,\cr
&|J_6|\le\varepsilon|u_x|_6^2+c/\varepsilon|\nabla v|_3^2|\nabla\varphi_x|_2^2,\cr
&|J_7|\le\varepsilon|u_x|_6^2+c/\varepsilon|\nabla v|_3^2|\nabla\varphi_x|_2^2,\cr
&|J_8|\le|\nabla\varphi_x|_3^2+|\nabla v|_6^2|\nabla\varphi_x|_2^2.\cr}
$$
Next we examine $L$. We express it in the form
$$\eqal{
&-\intop_\Omega(u-\nabla\varphi)\cdot\nabla u_x\cdot(u_x-\nabla\varphi_x)dx+\intop_\Omega(u-\nabla\varphi)\cdot\nabla v_x\cdot(u_x-\nabla\varphi_x)dx\cr
&=-\intop_\Omega(u-\nabla\varphi)\cdot\nabla u_x\cdot u_xdx+\intop_\Omega u\cdot\nabla u_x\cdot\nabla\varphi_xdx\cr
&\quad-\intop_\Omega\nabla\varphi\cdot\nabla u_x\cdot\nabla\varphi_xdx+\intop_\Omega u\cdot\nabla v_x\cdot u_xdx\cr
&\quad-\intop_\Omega u\cdot\nabla v_x\cdot\nabla\varphi_xdx-\intop_\Omega\nabla\varphi\cdot\nabla v_x\cdot u_xdx\cr
&\quad+\intop_\Omega\nabla\varphi\cdot\nabla v_x\cdot\nabla\varphi_xdx\equiv\sum_{i=1}^7L_i,\cr}
$$
where $L_1=0$,
$$\eqal{
&|L_2|\le\varepsilon|\nabla u_x|_2^2+c/\varepsilon|u|_6^2|\nabla\varphi_x|_3^2,\cr
&|L_3|\le\varepsilon|\nabla u_x|_2^2+c/\varepsilon|\nabla\varphi|_6^2|\nabla\varphi_x|_3^2,\cr
&|L_4|\le\varepsilon|u_x|_6^2+c/\varepsilon|u|_3^2|\nabla v_x|_2^2,\cr
&|L_5|\le|\nabla\varphi_x|_3^2+|u|_6^2|\nabla v_x|_2^2,\cr
&|L_6|\le\varepsilon|u_x|_6^2+c/\varepsilon|\nabla\varphi|_3^2|\nabla v_x|_2^2,\cr
&|L_7|\le|\nabla\varphi_x|_3^2+|\nabla\varphi|_6^2|\nabla v_x|_2^2.\cr}
$$
Finally, the last term on the l.h.s. of (\ref{5.8}) equals
$$
\mu|\nabla u_x|_2^2-\mu\intop_\Omega\nabla u_x\cdot\nabla^2\varphi_xdx,
$$
where the second integral is bounded by
$$
\varepsilon|\nabla u_x|_2^2+c/\varepsilon|\nabla^2\varphi_x|_2^2.
$$
Using the above estimates in (\ref{5.8}) yields
\begin{equation}\eqal{
&{d\over dt}|u_x|_2^2+\mu\|\nabla u\|_1^2\le c|u_x|_2^6+c\|u\|_1^2(|\rot\psi_x|_3^2+|\nabla\varphi_x|_3^2\cr
&\quad+|\nabla v|_3^2+|\nabla v_x|_2^2)+c(|\rot\psi_x|_3^2|\nabla\varphi_x|_2^2+ |\rot\psi|_\infty^2|\nabla\varphi_x|_2^2\cr
&\quad+|\nabla\varphi_x|_2^2|\nabla\varphi_x|_3^2+|\nabla^2\varphi_x|_2^2+ |\nabla\varphi_x|_3^2+ |\nabla\varphi|_6^2|\nabla\varphi_x|_3^2\cr
&\quad+|\nabla v|_6^2|\nabla\varphi_x|_2^2+|\nabla\varphi|_6^2|\nabla v_x|_2^2+|\nabla\varphi_{xt}|_2^2).\cr}
\label{5.9}
\end{equation}
Simplifying (\ref{5.9}) we obtain
\begin{equation}\eqal{
&{d\over dt}|u_x|_2^2+\mu\|\nabla u\|_1^2\le c|u_x|_2^6+c\|u\|_1^2\|v\|_2^2\cr
&\quad+c(\|v\|_2^2\|\nabla\varphi\|_1^2+|\nabla\varphi_{xt}|_2^2+ \|\nabla\varphi\|_2^2+|\nabla\varphi|_6^2|\nabla\varphi_x|_3^2\cr
&\quad+|\nabla\varphi_x|_2^2|\nabla\varphi_x|_3^2).\cr}
\label{5.10}
\end{equation}
This implies (\ref{5.7}) and concludes the proof.
\end{proof}

\begin{lemma}\label{l5.3}
Assume that Lemma \ref{l4.1} and Lemma \ref{l4.2} hold for any $t\in\R_+$. Assume that $cA^2\le\mu T/2$, $u(0)\in L_2(\Omega)$. Then
\begin{equation}\eqal{
|u(kT)|_2^2&\le c\exp(cA^2)(A^2+1){A^2\over\nu^2}/(1-\exp(-\mu T/2))\cr
&\quad+\exp(-\mu kT/2)|u(0)|_2^2,\quad k\in\N_0\cr}
\label{5.11}
\end{equation}
and for $t\in[kT,(k+1)T]$,
\begin{equation}\eqal{
|u(t)|_2^2&\le c\exp(cA^2)(A^2+1){A^2\over\nu^2}+\exp(-\mu(t-kT)+cA^2)\cdot\cr
&\quad\cdot\bigg[c\exp(cA^2)(A^2+1){A^2\over\nu^2}/(1-\exp(-\mu T/2))\cr
&\quad+\exp(-\mu kT/2)|u(0)|_2^2\bigg].\cr}
\label{5.12}
\end{equation}
\end{lemma}

\begin{proof}
In view of Lemma \ref{l4.1} it follows that Lemma \ref{l4.2} holds for interval $[kT,(k+1)T]$, $k\in\N_0$. Considering (\ref{5.1}) in the interval $[kT,(k+1)T]$ we have
\begin{equation}\eqal{
|u(t)|_2^2&\le\exp(-\mu t+cA^2)\intop_{kT}^t(\|v\|_1^2|\nabla\varphi|_6^2+ \|\nabla\varphi\|_1^2\cr
&\quad+|\nabla\varphi_t|_2^2)\exp(\mu t')dt'+\exp[-\mu(t-kT)+cA^2]|u(kT)|_2^2.\cr}
\label{5.13}
\end{equation}
Continuing, we have
\begin{equation}\eqal{
|u(t)|_2^2&\le c\exp(cA^2)(\|v\|_{1,\infty,\Omega_{kT}^t}^2 |\nabla\varphi|_{6,2,\Omega_{kT}^t}^2+\|\nabla\varphi\|_{1,2,\Omega_k^t}^2\cr
&\quad+|\nabla\varphi_t|_{2,\Omega_{kT}^t}^2)+\exp(-\mu(t-kT)+cA^2)|u(kT)|_2^2,\cr}
\label{5.14}
\end{equation}
where $\Omega_{kT}^t=\Omega\times(kT,t)$. Setting $t=(k+1)T$ and using the properties of solutions described by Lemma \ref{l4.2} we have
\begin{equation}\eqal{
|u((k+1)T)|_2^2&\le c\exp(cA^2)(A^2+1){A^2\over\nu^2}\cr
&\quad+\exp(-\mu T+cA^2)|u(kT)|_2^2.\cr}
\label{5.15}
\end{equation}
Using $-{\mu\over2}T+cA^2\le 0$ and the iteration we obtain from (\ref{5.15}) the inequality (\ref{5.11}). From (\ref{5.14}) we have
\begin{equation}
|u(t)|_2^2\le c\exp(cA^2)(A^2+1){A^2\over\nu^2}+\exp[-\mu(t-kT)+cA^2]|u(kT)|_2^2.
\label{5.16}
\end{equation}
From (\ref{5.16}) and (\ref{5.11}) we obtain (\ref{5.12}). This concludes the proof.
\end{proof}

\begin{lemma}\label{l5.4}
Let the assumptions of Theorem \ref{t3.6}, Lemmas \ref{l4.1}, \ref{l4.2} hold. Let $\|u(0)\|_1^2\le\gamma$, where $\gamma\le\gamma_*$ and $\gamma_*$ is so small that $\mu-c\exp(2cA^2)\gamma_*^2\ge\mu/2$.\\
Then for $\nu$ and $T$ sufficiently large we have
\begin{equation}
\|u(kT)\|_1^2\le\gamma,\quad k\in\N_0,
\label{5.17}
\end{equation}
and
\begin{equation}
\|u(t)\|_1^2\le\bigg[c(A^2+1){A^2\over\nu^2}+\gamma\bigg]\exp(cA^2),\quad t\in[kT,(k+1)T].
\label{5.18}
\end{equation}
\end{lemma}

\begin{proof}
From (\ref{5.1}) and (\ref{5.7}) we have
\begin{equation}\eqal{
&{d\over dt}\|u\|_1^2+\mu\|\nabla u\|_1^2\le c\|u\|_1^6+c\|u\|_1^2\|v\|_2^2+ (\|v\|_2^2\|\nabla\varphi\|_1^2\cr
&\quad+|\nabla\varphi_{xt}|_2^2+\|\nabla\varphi\|_2^2+|\nabla\varphi|_6^2 |\nabla\varphi_x|_3^2+|\nabla\varphi_x|_2^2|\nabla\varphi_x|_3^2).\cr}
\label{5.19}
\end{equation}
We write (\ref{5.19}) in the form
\begin{equation}\eqal{
{d\over dt}\|u\|_1^2&\le-(\mu-c\|u\|_1^4)\|u\|_1^2+c\|u\|_1^2\|v\|_2^2\cr
&\quad+c(\|v\|_2^2\|\nabla\varphi\|_1^2+|\nabla\varphi_{xt}|_2^2+ \|\nabla\varphi\|_2^2+|\nabla\varphi|_6^2|\nabla\varphi_x|_3^2\cr
&\quad+|\nabla\varphi_x|_2^2|\nabla\varphi_x|_3^2).\cr}
\label{5.20}
\end{equation}
We consider (\ref{5.20}) in the interval $[kT,(k+1)T]$, $k\in\N_0$. We know that Lemma \ref{l4.2} holds in this interval. Assume that
\begin{equation}
\|u(kT)\|_1\le\gamma,
\label{5.21}
\end{equation}
where $\gamma$ will be chosen sufficiently small. Introduce the quantity
\begin{equation}
X_0^2(t)=\exp\bigg(-c\intop_{kT}^t\|v(t')\|_2^2dt'\bigg)\|u(t)\|_1^2.
\label{5.22}
\end{equation}
Assume that estimate (\ref{3.53}) holds for interval $[kT,(k+1)T]$. Then
\begin{equation}
\intop_{kT}^t\|v(t')\|_2^2dt'\le A^2,\quad t\in[kT,(k+1)T].
\label{5.23}
\end{equation}
Therefore
\begin{equation}
X_0^2(t)\ge\exp(-cA^2)\|u(t)\|_1^2.
\label{5.24}
\end{equation}
Introduce the quantity
\begin{equation}\eqal{
G^2(t)&=c(\|v\|_2^2\|\nabla\varphi\|_1^2+|\nabla\varphi_{xt}|_2^2+ \|\nabla\varphi\|_2^2\cr
&\quad+|\nabla\varphi|_6^2|\nabla\varphi_x|_3^2+|\nabla\varphi_x|_2^2 |\nabla\varphi_x|_3^2).\cr}
\label{5.25}
\end{equation}
Let
\begin{equation}
K^2(t)=\exp\bigg(-c\intop_{kT}^t\|v(t')\|_2^2dt'\bigg)G^2(t).
\label{5.26}
\end{equation}
Then (\ref{5.20}) takes the form
\begin{equation}
{d\over dt}X_0^2\le-\bigg[\mu-c\exp\bigg(2c\intop_{kT}^t\|v(t')\|_{2^2}dt'\bigg)X_0^4\bigg] X_0^2+K^2.
\label{5.27}
\end{equation}
In view of (\ref{3.53}) we have
\begin{equation}
K^2(t)\le c(A^2+1){A^2\over\nu^2}
\label{5.28}
\end{equation}
and (\ref{5.22}) implies
\begin{equation}
X_0^2(kT)=\|u(kT)\|_1^2\le\gamma.
\label{5.29}
\end{equation}
Suppose that
$$t_*=\inf\{t\in[kT,(k+1)T]\colon X_0^2(t)>\gamma\}.
$$
Let $\gamma\in(0,\gamma_*]$, where $\gamma_*$ is so small that
\begin{equation}
\mu-c\exp(2cA^2)\gamma_*^4\ge\mu/2.
\label{5.30}
\end{equation}
Hence, for $t\le t_*$ we derive from (\ref{5.27}) the inequality
\begin{equation}
{d\over dt}X_0^2\le-{\mu\over 2}X_0^2+K^2.
\label{5.31}
\end{equation}
Assume that $\nu$ is so large that
$$
K^2(t)\le c(A^2+1){A^2\over\nu^2}\le{\mu\over 4}\gamma\quad {\rm for}\quad t\in[kT,(k+1)T].
$$
Then
$$
{d\over dt}X_0^2|_{t=t_*}\le-\bigg({\mu\over2}-{\mu\over4}\bigg)\gamma<0
$$
so $t_*$ does not exist in $[kT,(k+1)T]$. Hence
\begin{equation}
\|u(t)\|_1^2\le\gamma\exp(cA^2),\quad t\in[kT,(k+1)T].
\label{5.32}
\end{equation}
In view of (\ref{5.30}), (\ref{5.32}) and for $\gamma\le\gamma_*$ we obtain from (\ref{5.20}) the inequality
\begin{equation}\eqal{
{d\over dt}\|u\|_1^2&\le-{\mu\over2}\|u\|_1^2+c\|u\|_1^2\|v\|_2^2+c(\|v\|_2^2 \|\nabla\varphi\|_1^2\cr
&\quad+|\nabla\varphi_{xt}|_2^2+\|\nabla\varphi\|_2^2+|\nabla\varphi|_6^2 |\nabla\varphi_x|_3^2+|\nabla\varphi_x|_2^2|\nabla\varphi_x|_3^2).\cr}
\label{5.33}
\end{equation}
From (\ref{5.33}) we have
\begin{equation}\eqal{
&{d\over dt}\bigg[\|u\|_1^2\exp\bigg({\mu\over2}t-c\intop_{kT}^t\|v(t')\|_2^2dt'\bigg)\bigg]\le c\big(\|v\|_2^2\|\nabla\varphi\|_1^2+|\nabla\varphi_{xt}|_2^2\cr
&\quad+\|\nabla\varphi\|_2^2+ |\nabla\varphi|_6^2|\nabla\varphi_x|_3^2\cr
&\quad+|\nabla\varphi_x|_2^2|\nabla\varphi_x|_3^2\big)\exp \bigg({\mu\over2}t-c\intop_{kT}^t ||v(t'\|_2^2dt'\bigg)\bigg).\cr}
\label{5.34}
\end{equation}
Integrating (\ref{5.34}) with respect to time yields
\begin{equation}\eqal{
\|u(t)\|_1^2&\le c\exp\bigg(-{\mu\over2}t+c\intop_{kT}^t\|v(t')\|_2^2dt'\bigg) \intop_{kT}^t(\|v\|_2^2\|\nabla\varphi\|_1^2\cr
&\quad+|\nabla\varphi_{xt}|_2^2+\|\nabla\varphi\|_2^2+|\nabla\varphi|_6^2 |\nabla\varphi_x|_3^2\cr
&\quad+|\nabla\varphi_x|_2^2|\nabla\varphi_x|_3^2)\exp\bigg({\mu\over2}t'\bigg)dt'\cr
&\quad+\exp\bigg(-{\mu\over2}(t-kT)+c\intop_{kT}^t\|v(t')\|_2^2dt'\bigg) \|u(kT)\|_1^2.\cr}
\label{5.35}
\end{equation}
Setting $t=(k+1)T$ we derive
\begin{equation}\eqal{
\|u((k+1)T)\|_1^2&\le c\exp(cA^2)(A^2+1){A^2\over\nu^2}\cr
&\quad+\exp\bigg(-{\mu\over2}T+cA^2\bigg)\|u(kT)\|_1^2.\cr}
\label{5.36}
\end{equation}
Assuming that $-{\mu\over4}T+cA^2\le 0$, $\|u(kT)\|_1^2\le\gamma$ we obtain
$$
\|u((k+1)T)\|_1^2\le c\exp(cA^2)(A^2+1){A^2\over\nu^2}+\exp\bigg(-{\mu\over4}T\bigg)\gamma\le\gamma
$$
which holds for sufficiently large $\nu$ and $T$. Hence
\begin{equation}
\|u(kT)\|_1^2\le\gamma\quad {\rm for\ any}\ \ k\in\N_0.
\label{5.37}
\end{equation}
Using (\ref{5.37}) in (\ref{5.35}) yields
\begin{equation}
\|u(t)\|_1^2\le\bigg[c(A^2+1){A^2\over\nu^2}+\gamma\bigg]\exp(cA^2),
\label{5.38}
\end{equation}
where $t\in[kT,(k+1)T]$. Estimates (\ref{5.37}) and (\ref{5.38}) imply (\ref{5.17}) and (\ref{5.18}). This concludes the proof.
\end{proof}

\section{Estimate for solutions to the Navier-\\ -Stokes equations}\label{s6}

\begin{theorem}\label{t6.1}
Let the assumptions of Theorem \ref{t3.6}, Lemmas \ref{l4.1}, \ref{l4.2}, \ref{l5.3}, \ref{l5.4} hold. Then
\begin{equation}
\|V(t)\|_1^2\le A^2+\bigg[c(A^2+1){A^2\over\nu^2}+\gamma\bigg]\exp(cA^2).
\label{6.1}
\end{equation}
\end{theorem}

\begin{proof}
Recall that
$$
V=v-u
$$
Hence using estimates (\ref{3.40}), (\ref{5.18}) yields
$$
\|V(t)\|_1^2\le\|v\|_1^2+\|u\|_1^2\le A^2+\bigg[c(A^2+1){A^2\over\nu^2}+\gamma\bigg]\exp(cA^2).
$$
This concludes the proof.
\end{proof}

\bibliographystyle{amsplain}

\end{document}